\documentclass[10pt]{article}   
\usepackage{amsfonts,amsmath,latexsym}
\usepackage[T1]{fontenc}
\usepackage{dsfont}
\usepackage{comment}
\usepackage{hyperref}
\usepackage[dvips]{graphicx}
\usepackage{epsfig}
\usepackage{enumerate}
\usepackage{epsf}   
\usepackage{amsthm}
\usepackage{color}
\usepackage{amssymb}
\usepackage{comment}
\usepackage{caption}
\usepackage{fancyhdr} 

\usepackage{palatino}

\newtheorem{thm}{Theorem}[section]
\newtheorem{prop}[thm]{Proposition}
\newtheorem{lem}[thm]{Lemma}
\newtheorem{corro}[thm]{Corollary}
\newtheorem{defi}[thm]{Definition}
\newtheorem{rem}[thm]{Remark}

\newtheorem{ass}{Assumption}
\newtheorem{property}{Property}
\newtheorem{notation}{Notation}

\newenvironment{prooff}[1]{\begin{trivlist}
\item {\it \bf Proof}\quad} {\qed\end{trivlist}}

\def\R{\mathbb R}
\def\N{\mathbb N}

\def\E{\mathbb E}
\def\P{\mathbb P}
\def\Q{\mathbb Q}
\def\sha{{\cal A}}

\def\shc{{\cal C}}
\def\shd{{\cal D}}
\def\she{{\cal E}}
\def\shf{{\cal F}}
\def\shg{{\cal G}}
\def\shm{{\cal M}}

\def\shl{{\cal L}}
\def\shp{{\cal P}}
\def\shs{{\cal S}}

\def \be{\begin{equation*}}
\def \ee{\end{equation*}}
\def\restriction#1#2{\mathchoice
              {\setbox1\hbox{${\displaystyle #1}_{\scriptstyle #2}$}
              \restrictionaux{#1}{#2}}
              {\setbox1\hbox{${\textstyle #1}_{\scriptstyle #2}$}
              \restrictionaux{#1}{#2}}
              {\setbox1\hbox{${\scriptstyle #1}_{\scriptscriptstyle #2}$}
              \restrictionaux{#1}{#2}}
              {\setbox1\hbox{${\scriptscriptstyle #1}_{\scriptscriptstyle #2}$}
              \restrictionaux{#1}{#2}}}
\def\restrictionaux#1#2{{#1\,\smash{\vrule height .8\ht1 depth .85\dp1}}_{\,#2}} 

\author{
{\sc Lucas IZYDORCZYK}
\thanks{ENSTA Paris, Institut Polytechnique de Paris.
Unit\'e de Math\'ematiques Appliqu\'ees (UMA).
 E-mail:{ \tt lucas.izydorczyk@ensta-paris.fr}} 
{\sc,}\ {\sc Nadia OUDJANE}
\thanks{EDF R\&D,   and FiME (Laboratoire de Finance des March\'es de l'Energie
(Dauphine, CREST,  EDF R\&D) www.fime-lab.org). 
E-mail:{\tt  
nadia.oudjane@edf.fr}}
 {\sc, Francesco RUSSO,} 
\thanks{ENSTA Paris, Institut Polytechnique de Paris.
Unit\'e de Math\'ematiques Appliqu\'ees (UMA).
E-mail (corresponding author): {\tt russo@math.univ-paris13.fr} }
 \\
{\sc and}\ {\sc Gianmario TESSITORE}
\thanks{Universit\`a Milano-Bicocca.
Dipartimento di Matematica.
E-mail:{\tt  gianmario.tessitore@unimib.it}
}}

\date{September 5th 2021}

\title{Fokker-Planck equations with terminal
condition and related McKean probabilistic representation}

\newcommand{\MBFigure}[6]{
$\left. \right.$ \\
\refstepcounter{figure}
\addcontentsline{lof}{figure}{\numberline{\thefigure}{\ignorespaces #5}}
\begin{center}
\begin{minipage}{#1cm}
\centerline{\includegraphics[width=#2cm,angle=#3]{#4}}
\begin{center}
\upshape{F\textsc{ig} \normal
\end{center}
size{\thefigure}. $-$} #5
\end{center}
\label{#6}
\end{minipage}
\end{center}
$\left. \right.$ \\}

\begin{document}
\maketitle

\begin{abstract}
  Usually Fokker-Planck type partial differential equations (PDEs)
 are well-posed if the initial condition
  is specified. In this paper, alternatively, we consider the inverse
  problem which consists in 
  prescribing final data: in particular we give sufficient conditions
  for uniqueness.
  In the second part of the paper we provide a probabilistic
  representation of those PDEs in the form of
  a solution of a McKean type equation
  corresponding to the time-reversal dynamics of a diffusion process.
\end{abstract}


\medskip\noindent  {\bf 2020  AMS-classification}.
60H10; 60H30; 60J60; 35R30.
 
 
\section{Introduction}

\medskip

The main objective of the paper consists in studying well-posedness
and probabilistic representation of the  Fokker-Planck PDE
with terminal condition
\begin{equation} \label{EDPTerm0}
\left \{
\begin{array}{lll}
\partial_t {\bf u} &=& \frac{1}{2} 
\displaystyle{\sum_{i,j=1}^d} \partial_{ij}^2 \left( (\sigma \sigma^{\top})_{i,j}(t,x) {\bf u} \right) - div \left( b(t,x) {\bf u} \right)\\
{\bf u}(T) &=& \mu,
\end{array}
\right .
\end{equation}
where $\sigma: [0,T] \times \R^d  \rightarrow  M_{d}(\R)$,
$b: [0,T] \times \R^d  \rightarrow  \R^d$ and
$\mu$ is a prescribed finite Borel (most often non-negative)  measure on $\R^d$.
When ${\bf u}(t)$ admits a density for some $t \in [0,T]$
we write  ${\bf u}(t) = u(t,x) dx$.
This equation is motivated  
by applications in various domains of physical sciences and engineering,
as heat 
conduction~\cite{beck1985inverse},
 material science~\cite{ renardy1987mathematical} or
 hydrology~\cite{bagtzoglou2003marching}.
 In particular, \textit{hydraulic inversion} is interested in inverting a diffusion process representing the concentration of a pollutant to identify the pollution source location when the final concentration profile is observed.
 Those models give often rise to ill-posed problems because,
 either the solution is not unique or
it is not stable.
In this specific case, the existence is ensured by
 the fact  that the observed contaminant is necessarily originated
  at a given time (as soon as the model is correct). 
Several authors have handled the lack of uniqueness
 by introducing regularization methods and
approaching the problem using well-posed PDEs, see typically
\cite{tikhonov1977solutions}  and \cite{lattes1969method}.
In particular for the PDE \eqref{EDPTerm0} there are very few results
even concerning existence and uniqueness.
The first objective of the paper is precisely to investigate
uniqueness for \eqref{EDPTerm0}.
The second objective is to propose a probabilistic representation of
PDE \eqref{EDPTerm0}. Our approach relies  on the
existence and uniqueness for that PDE. 
Although it is beyond the scope of this paper, it is important to emphasize
the interests of probabilistic representation in possibly bringing new insights in stability analysis or numerical approximation of PDE \eqref{EDPTerm0}. 
For instance, based on probabilistic representation of nonlinear PDEs \cite{bossytalay1, JourMeleard} have  developed stochastic
 particle methods in the spirit of McKean to provide original Monte Carlo approximation schemes approaching several class of PDEs.
For recent contributions in that direction, one can refer to
\cite{LOR1,LOR2,LOR3,LOR4} and the survey paper 
\cite{papierUgolini}.
In the same spirit, one may develop Monte Carlo approximation schemes for PDE \eqref{EDPTerm0} based on the probabilistic representation provided in the present paper, which  will be the object of future works. 
Besides, the probabilistic representation of PDE \eqref{EDPTerm0} has already been exploited in \cite{IOR_Controle}, in the specific setting of Gaussian diffusions to propose an original approximation scheme for solving semi-linear PDEs with applications to stochastic control.  
 

To realize the probabilistic representation of the PDE 
\eqref{EDPTerm0}, when $\mu$ is non-negative, 
we consider the renormalized PDE
\begin{equation} \label{EDPTerm}
\left \{
\begin{array}{lll}
\partial_t {\bf \bar u} &=& \frac{1}{2} 
\displaystyle{\sum_{i,j=1}^d} \partial_{ij}^2 \left( (\sigma \sigma^{\top})_{i,j}(t,x) {\bf \bar u} \right) - div \left( b(t,x) {\bf \bar u} \right)\\
{\bf \bar u}(T) &=& \bar \mu,
\end{array}
\right.
\end{equation}
where $\bar \mu = \frac{\mu}{\mu(\R^d)}$ is a probability measure.
  We remark that the PDEs \eqref{EDPTerm} and \eqref{EDPTerm0} are
  equivalent in the sense that a solution to \eqref{EDPTerm} (resp.
   \eqref{EDPTerm0}) provides a solution to the other one. 
  The program consists in  considering
the  McKean
type stochastic differential equation (SDE)
\begin{equation}\label{MKIntro}
\begin{cases}
\displaystyle Y_t = Y_0 - \int^{t}_{0}b\left(T-r,Y_r\right)dr +
 \int^{t}_{0}\left\{\frac{\mathop{div_y}\left(\Sigma_{i.}\left(T-r,Y_r\right)p_{r} \left(Y_r\right)\right)}{p_{r}\left(Y_r\right)}\right\}_{i\in[\![1,d]\!]}dr + \int^{t}_{0} \sigma\left(T-r,Y_r\right)d\beta_r, \\
p_t \ \rm{density \ of} \ {\bf p}_t =  \rm{law\ of} \ Y_t, t \in ]0,T[,
\\
Y_0 \sim
\bar \mu,
\end{cases}
\end{equation}
 where $\beta$ is a $d$-dimensional Brownian motion and $\Sigma = \sigma \sigma^\top$, whose solution is the couple $(Y,{\bf p})$. 
Indeed an application of It\^o formula (see Proposition \ref{PProbRep})
 shows that whenever 
$(Y,{\bf p})$ is a solution of the SDE \eqref{MKIntro} then 
$t \mapsto {\bf p}_{T-t}$
is a solution of \eqref{EDPTerm}. 

The idea of considering \eqref{MKIntro} comes from the
SDE verified 
by time-reversal of a diffusion.
Time-reversal of Markov processes was explored by several authors:
see for instance
 \cite{haussmann_pardoux} for the diffusion case in finite dimension,
\cite{wakolbinger}
 for the diffusion case in infinite dimension  
 and \cite{jacod_levy} for the jump case.
 We also mention the two very interesting recent preprints \cite{cattiaux2021time, conforti2021time} in relation with entropy.

Consider a \textit{forward} diffusion process $X$ solution of
\begin{equation}
\label{eq:X} 
X_t=X_0+\int_0^t b(s,X_s)ds+\int^{t}_{0}\sigma(s,X_s)dW_s, \ t \in [0,T],
\end{equation}
where $\sigma$ and $b$ are Lipschitz coefficients  with linear growth
and $W$ is a standard Brownian motion on $\R^d$.
$X$ is a probabilistic representation of 
\begin{equation} \label{EDPInitial}
\left \{
\begin{array}{lll}
\partial_t {\bf u} &=& \frac{1}{2} 
\displaystyle{\sum_{i,j=1}^d} \partial_{ij}^2 \left( (\sigma \sigma^{\top})_{i,j}(t,x) {\bf  u} \right) - div \left( b(t,x) {\bf  u} \right)\\
{\bf  u}(0) &=&  \nu,
\end{array}
\right.
\end{equation}
where $X_0 \sim \nu$. Indeed, whenever $X$ is
a solution of the SDE \eqref{eq:X} then the function $t \mapsto  {\bf u}(t)$,
where $ {\bf u}(t)$ is the law of $X_t$ is a solution in the sense of
distributions of the PDE \eqref{EDPInitial}.
We remark also that $t \mapsto  {\bf u}(t)$ solves 
the PDE \eqref{EDPTerm0}, $\mu$ being the law of $X_T$. 
Let us now denote $\hat X_t:= X_{T-t}, t \in [0,T]$  the time-reversal process
of the solution $X$ of  \eqref{eq:X}.
In \cite{haussmann_pardoux}
the authors gave sufficient general conditions on $\sigma, b$ and on the marginal laws  $p_t$ of $X_t$
so that $Y:= \hat X$
is a  solution (in law) of the SDE
\begin{equation}\label{IntroPardoux}
\displaystyle Y_t = X_T - \int^{t}_{0}b\left(T-r,Y_r\right)dr +
\int^{t}_{0}\left\{\frac{\mathop{div_y}\left(\Sigma_{i.}\left(T-r,Y_r\right)p_{T-r} \left(Y_r\right)\right)}{p_{T-r}\left(Y_r\right)}\right\}_{i\in[\![1,d]\!]}dr + \int^{t}_{0} \sigma\left(T-r,Y_r\right)d\beta_r.
\end{equation}
This constitutes an essential tool that we will exploit to prove existence of the McKean SDE  
\eqref{MKIntro}.

As far as uniqueness for \eqref{MKIntro} is concerned,
we repeat that the key idea relies on uniqueness for the PDE  \eqref{EDPTerm}
(or \eqref{EDPTerm0}).
First of all Proposition
\ref{PProbRep},
states the following.
If $\left(Y,{\bf p}\right)$
is a solution of \eqref{MKIntro}, then ${\bf p}\left(T-\cdot\right)$ is a solution
of the PDE \eqref{EDPTerm0}, with $\mu = {\bf p}(0)$.
This fact justifies the terminology that \eqref{MKIntro}
constitutes a probabilistic representation of
\eqref{EDPTerm0}.
Now, if the PDE  \eqref{EDPTerm0} admits at most one solution then
${\bf p}$ is completely identified, so  \eqref{MKIntro}
reduces to an ordinary SDE for which uniqueness in law (resp. pathwise)
can be established whenever the coefficients are shown to
be locally bounded (resp. locally Lipschitz).

As  we have mentioned earlier, there are
not  many articles analyzing uniqueness for Fokker-Planck PDEs with terminal condition.
For introductory purposes, we present two simple situations
when this problem can be easily tackled: one by analytical
means and one by probabilistic techniques.
\begin{description}

\item{a)}  The heat equation with
 terminal condition admits uniqueness.
Suppose indeed that $u : [0,T] \mapsto \mathcal{S}'\left(\R^d\right)$
solves
\begin{equation} \label{HeatPDE}
\begin{cases}
\partial_t{\bf u} = \Delta {\bf u} \\
        {\bf u}\left(T\right) =  \mu.
\end{cases}
\end{equation}
 Then, the Fourier transform of $u$, $v\left(t,\cdot\right) :=
\mathcal{F}{\bf u}\left(t,\cdot\right), t \in [0,T]$ solves the  ODE
(for fixed $\xi \in \R^d$)
\begin{equation} \label{HeatODE}
\begin{cases}
  \frac{d}{dt}v\left(t,\xi\right) =
  - \left|\xi\right|^2v\left(t,\xi\right), \left(t,\xi\right) \in [0,T]\times\R^d\\
v\left(T,\cdot\right) = \mathcal{F}\mu.
\end{cases}
\end{equation}
 This admits at most one solution, since setting $\mathcal{F}\mu = 0$ the unique solution of \eqref{HeatODE} is the null function. 
\item{b)} Another relatively simple situation is described below
to study uniqueness among the solutions of
the PDE \eqref{EDPTerm0} whose initial value belongs to the class of Dirac measures.
Consider the example when $\sigma$ is continuous bounded non-degenerate 
and the drift $b$ is affine i.e. $b(s,y) = b_0\left(s\right) + b_1\left(s\right) y, \left(s,y\right) \in [0,T]\times\R^d$, $b_0$ (resp. $\ b_1$) being mappings from $[0,T]$ to $\R^d$ (resp. to ${M}_d\left(\R\right)$).
 Suppose for a moment that the PDE in the first line of \eqref{EDPTerm0}, but with initial condition
(see \eqref{Fokker}) is well-posed.
Sufficient conditions for this
will be provided in Remark \ref{R1}.

Let  $x  \in \R^d$ and $u$ be a solution of the PDE  \eqref{EDPTerm0}
such that $u(0,\cdot) = \delta_{x}$. 
If $X^{x}$ is
the solution of \eqref{eq:X} with initial condition  $x$,
it is well-known that the family of laws of
$X^{x}_t, t \in [0,T]$, is a solution 
of \eqref{EDPTerm0}.
So this coincides with $u(t,\cdot)$
and in particular $ \mu$ is the law of $X^{x}_T$. 
To conclude we only need to determine $x$.
Taking the expectation in the SDE fulfilled by $X^{x}$, we show that
the function $t \mapsto E^x(t) := \E(X^{x}_t)$ is solution
of 
$$ E^x(t) = \int_{\R^d} y \mu\left(dy\right) - \int_t ^T \left(b_0(s) + b_1(s) E^x(s)\right) ds.$$ 
Previous linear ODE has clearly a unique solution. At this point
$x = E(0)$ is uniquely determined.
\end{description}
Those examples give a flavor of how to tackle the uniqueness
issue for the PDE \eqref{EDPTerm0}. However, generalizing those approaches is far more complicated 
and constitutes the first part  of the present work.
The contributions of the paper are twofold.
\begin{enumerate}
\item We investigate uniqueness for the Fokker-Planck PDE with
terminal condition \eqref{EDPTerm0}. This
is  done in Section \ref{S3} 
in two different situations: the case when
the coefficients are bounded and the situation
of a PDE associated with an inhomogeneous
Ornstein-Uhlenbeck (OU) semigroup.
In  Section \ref{S32}
we show uniqueness for bounded continuous coefficients
when solutions start in the class $\shc$
of multiples of Dirac measures.
In Proposition \ref{propLip1} we discuss 
dimension $d = 1$.
Theorem \ref{propLipd} is devoted to the case $d \ge 2$.
We distinguish the non-degenerate case from the possibly
degenerate case but with smooth coefficients proving
uniqueness for small time horizon $T$.
In Section \ref{SGP} we show uniqueness
when the coefficients are
stepwise time-homogeneous.
In Theorem \ref{P315} the coefficients are time-homogeneous,
bounded and H\"older, with non-degenerate diffusion.
Corollary \ref{C313} extends previous results to
the case of stepwise time-inhomogeneous coefficients.
In Section \ref{S34}, Theorem \ref{BwdOU_Uniq} treats
the Ornstein-Uhlenbeck case.
\item We study existence and uniqueness in law
  for the McKean SDE \eqref{MKIntro}, with some specific
  remarks concerning strong existence and pathwise uniqueness.
After some preliminary considerations in Section \ref{Prelim}, 
  Proposition  \ref{TExUniq} and Theorem
\ref{TC313} 
discuss the case of bounded coefficients.
Theorem \ref{MKOU_WellP}  is devoted to the case of Ornstein-Uhlenbeck
  (with not necessarily Gaussian terminal condition),
  where strong existence and pathwise uniqueness are established.
\end{enumerate}


\section{Notations and preliminaries}
\label{SNotations}

\setcounter{equation}{0}

 Let us fix $d \in \mathbb{N}^*$, $T > 0$.
   $\mathcal{C}^{\infty}_c\left(\R^d\right)$ is the linear space of smooth
 functions with compact support.
  For a given $p \in \N^*$, $[\![1,p]\!]$ denotes the set of all integers between $1$ and $p$ included. 
 $M_d\left(\R\right)$ stands for the set of $d\times d$ matrices. 
 $\left<,\right>$  denotes the usual scalar product on $\R^d$,
 with associated norm $\left|.\right|$.
 For a given $A \in M_d\left(\R\right)$, $Tr\left(A\right)$ (resp. $A^{\top}$) symbolizes the trace (resp. the transpose) of the matrix $A$. $\left|\left|A\right|\right|$ denotes the usual Frobenius norm. 
 We also  introduce the function $Jf$ from $\R^p$ to $M_{d}\left(\R\right)$ such that
$ Jf : z \mapsto \left(\partial_{j}f^i\left(z\right)\right)_{\left(i,j\right)\in
 [\![1,d]\!]\times[\![1,d]\!]}.
$

Let $\alpha \in ]0,1[, n \in \N$.
$\shc_b(\R^d)$ (resp. $\shc^n_b(\R^d)$) indicates the space
of bounded continuous functions 
(resp. bounded functions of class $\shc^n$ such that
all the derivatives are bounded).
$ \mathcal{C}^{\alpha}(\R^d), 0 < \alpha <1,$ is the Banach space of 
bounded $\alpha$-H\"older functions $\R^d \rightarrow \R$
equipped with the norm
$
\left|.\right|_{\alpha} :=  \left|\left|.\right|\right|_{\infty} + 
\left[.\right]_{\alpha}, 
$
where  
$$ \left[f\right]_{\alpha} := 
\sup_{x,y \in \R^d, x \neq y}
\frac{\left|f(x) - f(y)\right|}{\left|x-y\right|^{\alpha}} < 
\infty$$
and $\Vert \cdot \Vert_\infty$ is the sup-norm.
If  $n$ is some integer
$ \mathcal{C}^{\alpha+n}(\R^d)$ is the Banach space of bounded functions 
 $f: \R^d \rightarrow \R$  such that all its derivatives up to order
$n$ are bounded
and such that the derivatives of order $n$ are $\alpha$-H\"older continuous.
This is equipped with the norm obtained as the sum
of the $C^n_b(\R^d)$-norm plus the sum of the quantities $[g]_\alpha$ 
where $g$ is an $n$-order derivative of $f$. 
For more details, see Section 0.2 of \cite{lunardi_1995}.
If $E$ is a linear Banach space,
we denote by $\left|\left|.\right|\right|_{E}$ the associated operator norm and by $\mathcal{L}\left(E\right)$ the space of linear bounded operators 
$E \rightarrow E$. Often in the sequel we will have
$E = \mathcal{C}^{2\alpha}(\R^d)$.

$\shp\left(\mathbb{R}^d\right)$ (resp. $\shm_+\left(\mathbb{R}^d\right),
\shm_{f}\left(\mathbb{R}^d\right)$)
 denotes the set of probability 
(resp. non-negative finite valued, finite signed) measures
on $\left(\mathbb{R}^d,\mathcal{B}\left(\mathbb{R}^d\right)\right)$.
 We also denote by $\shs\left(\R^d\right)$ the space of Schwartz functions and by $\shs'\left(\R^d\right)$ the space of tempered distributions. For all $\phi \in \shs\left(\R^d\right)$ and $\mu \in \shm_f\left(\R^d\right)$, we set the notations
\be
\shf \phi : \xi \mapsto \int_{\R^d}e^{-i\left<\xi,x\right>}\phi\left(x\right)dx, \ \shf \mu : \xi \mapsto \int_{\R^d}e^{-i\left<\xi,x\right>}\mu\left(dx\right).
\ee
\smallbreak
\noindent Given a mapping ${\bf u}: [0,T] \rightarrow \mathcal{M}_f\left(\mathbb{R}^d\right)$, we convene that when for $t \in [0,T]$,  ${\bf u}\left(t\right)$ has a density, this is denoted by $u\left(t,\cdot\right)$. 
 Recalling $\Sigma = \sigma \sigma^\top$, let us introduce, for a given $t$ in $[0,T]$, the differential operator, 
\begin{equation} \label{EqOpL}
L_t f := \frac{1}{2}\sum^{d}_{i,j=1}\Sigma_{ij}(t,\cdot) \partial_{ij} f +
  \sum^{d}_{i=1 }b_i\left(t,\cdot\right)\partial_{i} f,
\end{equation}
$f \in C^2(\R^d)$ 
 and denote by $L^*_t$ its formal adjoint, which means
that for a given signed measure $\eta$
\begin{equation} \label{EqOpL*}
L^*_t \eta := \frac{1}{2} 
\displaystyle{\sum_{i,j=1}^d} \partial_{ij}^2 \left( \Sigma_{i,j}(t,x) \eta
\right) - div \left( b(t,x) \eta \right )\ .
\end{equation}
With this notation, the PDE \eqref{EDPTerm0} rewrites
\begin{equation} \label{BackwardFokker}
\begin{cases}
\partial_t{\bf u} = L^*_t{\bf u} \\
{\bf u}\left(T\right) = \mu. \\
\end{cases}
\end{equation}

 In the sequel we will often make use of the following assumptions.
\begin{ass}\label{Lip1d}
$b,\sigma$ are Lipschitz in space uniformly in time, with linear growth.
\end{ass} 
\begin{ass} \label{Zvon1a}
$b$ and $\Sigma$ are bounded.
\end{ass}
\begin{ass} \label{Zvon1b}
 $\Sigma$ is continuous. 
\end{ass}
\begin{ass} \label{Zvon3}
There exists $\epsilon > 0$ such that for all $t \in [0,T]$, $\xi \in \R^d$, $x \in \R^d$
\begin{equation}
\left<\Sigma(t,x)\xi,\xi\right> \geq \epsilon \left|\xi\right|^2.
\end{equation}
\end{ass}
 For a given random variable $X$ on a probability space $\left(\Omega,\shf,\P\right)$, $\mathcal{L}_{\P}\left(X\right)$
 denotes its law under $\P$ and
$\mathbb{E}_{\P}\left(X\right)$ its expectation under $\P$. When self-explanatory, the subscript will be omitted in the sequel.

\section{A Fokker-Planck PDE with terminal condition}
\label{S3}

\setcounter{equation}{0}

\subsection{Preliminary results on uniqueness}

 In this section, we consider a Fokker-Planck type PDE with terminal condition for which the notion of solution is clarified in the following definition.
\begin{defi} \label{Def}
\noindent Fix $\mu \in \shm_f\left(\R^d\right)$. We say that a mapping ${\bf u}$ from $[0,T]$ to
 $\shm_f\left(\R^d\right)$ solves the PDE~\eqref{EDPTerm0}, 
 if for all $\phi \in \mathcal{C}^{\infty}_c\left(\R^d\right)$ and all $t \in [0,T]$
\begin{equation} \label{weak}
\int_{\mathbb{R}^d}\phi\left(y\right){\bf u}\left(t\right)\left(dy\right) = 
\int_{\mathbb{R}^d}\phi\left(y\right)\mu\left(dy\right) - \int^{T}_{t}\int_{\mathbb{R}^d}L_s\phi\left(y\right){\bf u}\left(s\right)\left(dy\right)ds.
\end{equation}
\end{defi}
We consider the following property related to a
given class $\mathcal{C} \subseteq \shm_+\left(\R^d\right)$.
Later we will establish uniqueness results for \eqref{EDPTerm0}
provided that the solution starts in  $\shc$.

\begin{property} \label{GH1}
For all $\nu \in \mathcal{C}$, the PDE
\begin{equation}\label{Fokker}
\begin{cases}
\partial_t{\bf u} = L^*_t{\bf u} \\
{\bf u}\left(0\right) = \nu 
\end{cases}
\end{equation}
 admits at most one solution ${\bf u}: [0,T] \rightarrow \shm_+\left(\mathbb{R}^d\right)$.
\end{property}
 We recall that, for a given $\nu \in \shm_f\left(\R^d\right)$, ${\bf u}: [0,T] \rightarrow \shm_f\left(\mathbb{R}^d\right)$
is a solution of the PDE \eqref{Fokker} if for all $\phi \in \mathcal{C}^{\infty}_c\left(\R^d\right)$ and all $t \in [0,T]$,
\begin{equation} \label{weakbis}
\int_{\mathbb{R}^d}\phi\left(y\right){\bf u}\left(t\right)\left(dy\right) = \int_{\mathbb{R}^d}\phi\left(y\right)\nu\left(dy\right) + \int^{t}_{0}\int_{\mathbb{R}^d}L_s\phi\left(y\right){\bf u}\left(s\right)\left(dy\right)ds.
\end{equation}
 Suppose there is an $\shm_+\left(\R^d\right)$-valued  
 solution  ${\bf u}$  of \eqref{Fokker}
such that
${\bf u}(0) \in \shc$  for
  some class $\shc$. We also suppose
 that Property \ref{GH1} holds with respect to $\shc$.
Then this unique solution will be denoted by ${\bf u}^{\nu}$ in the sequel.
 We remark that, whenever Property \ref{GH1} holds with respect to a given $\shc \subseteq \shp\left(\R^d\right)$, then the PDE \eqref{Fokker} admits at most one $\shm_+\left(\R^d\right)$-valued solution with any initial value belonging to $\R^*_+\shc := \left(\alpha\nu\right)_{\alpha > 0,\nu \in \shc}$.
\smallbreak

We start with a simple but fundamental observation.
\begin{prop} \label{PFundam}
 Let us suppose $\sigma, b$ to be locally bounded, $\nu$
 be a Borel probability on $\R^d$,
$\alpha >  0$,
 $\xi$ be a r.v.
distributed according to $\nu$.
 Suppose that there is a solution $X$ of SDE
\begin{equation} \label{EqLin}
X_t = \xi +  \int^{t}_{0}b\left(r,X_r\right)dr + \int^{t}_{0}\sigma\left(r,X_r\right)dW_r, \  t \in [0,T], \ \mathbb{P}\rm{-a.s.},
\end{equation}
 where $W$ is a $d$-dimensional standard Brownian motion.
  Then the $\shm_+\left(\R^d\right)$-valued function
$t \mapsto \alpha\mathcal{L}\left(X_t\right)$ 
is a solution of the PDE \eqref{Fokker} with initial value $\alpha \nu$.     
\end{prop}

\begin{proof} \
One first applies It\^o formula to $\varphi(X_t)$,
where $\varphi$ is a smooth function with compact support
and then one takes the expectation.
\end{proof}
\begin{rem} \label{R1}
\begin{enumerate}
\item Suppose that the coefficients  $b,\Sigma$ are bounded.
 Property \ref{GH1} holds with respect to $\mathcal{C} :=
   \shm_+\left(\R^d\right)$
  as soon as
  the martingale problem associated with $b,\Sigma$ 
  admits uniqueness for all initial condition of the type $\delta_x, x \in \R^d$.
  Indeed, this is a consequence of Lemma 2.3 in \cite{figalli}.
\item Suppose $b$ and $\sigma$ with linear growth. Let
  $\nu \in \shm_+\left(\R^d\right)$ not trivially null  (resp. 
$\nu \in \shp\left(\R^d\right)$).
  By Proposition \ref{PFundam}, the existence of an $\shm_+\left(\R^d\right)$-valued (resp. $\shp\left(\R^d\right)$-valued) solution
  for the PDE \eqref{Fokker} (even on $t \ge 0$)
  is  ensured when the martingale problem associated to
  $b$ and $\Sigma$ admits existence (and consequently when the SDE 
\eqref{EqLin} admits weak existence) with initial condition
 $\frac{\nu}{\Vert\nu\Vert}$.
We remark that, for example, this happens when the coefficients $b, \sigma$
are continuous with linear growth: see Theorem 12.2.3 in \cite{stroock} 
for the case of 
bounded coefficients, the unbounded case can be easily obtained by truncation.
\item The martingale problem associated to $b$ and $\Sigma$ 
  is well-posed for all deterministic
  initial condition, for instance in the following cases.
\begin{itemize}
\item When $\Sigma, b$ have linear growth 
  and $\Sigma$   is continuous and non-degenerate
  (i.e. Assumptions \ref{Zvon1a} and \ref{Zvon3} hold),
see \cite{stroock} Corollary 7.1.7 and Theorem 10.2.2.
\item Suppose $d=1$ and $\sigma $ is bounded.
When $\sigma$ is lower bounded by a positive constant
  on each compact set, see \cite{stroock}, Exercise 7.3.3.
\item When $d =2$, $\Sigma$ is non-degenerate and
$\sigma$  and $b$ are time-homogeneous and bounded,
see \cite{stroock}, Exercise 7.3.4.
\item When $\sigma, b$ are Lipschitz with linear
growth  (with respect to the
space variable); in this case one obtains even strong solutions
of the corresponding stochastic differential equation.
\end{itemize}
\end{enumerate}
The lemma below provides in particular sufficient conditions for
 the validity of Property \ref{GH1}.

\begin{lem} \label{LC313}
\begin{enumerate} 
\item  Let  $\nu \in \shp\left(\R^d\right)$.
We suppose
 Assumptions \ref{Zvon1a}, \ref{Zvon1b}
 and \ref{Zvon3}.
Then there  is a unique $\shm_+\left(\R^d\right)$-valued 
 solution
${\bf u}$
 to the PDE \eqref{Fokker}
with  ${\bf u}(0) = \nu$.
Moreover ${\bf u}^\nu$ takes values in $\shp(\R^d)$.
In particular Property \ref{GH1}
related to the class $\shc = \shp(\R^d)$ is verified.
\item Under Assumptions \ref{Lip1d} and \ref{Zvon1a},
 Property \ref{GH1} is fulfilled for $\shc = \shm_+\left(\R^d\right)$.
\end{enumerate}

\end{lem}
\begin{proof} \
\begin{enumerate}
\item  Existence follows by items 2. and 3. of Remark \ref{R1}.
  Uniqueness is a consequence of items 1. and 3. of
  the same Remark.
\item  Since $b$ and $\sigma$ are Lipschitz,
 Property \ref{GH1} is fulfilled, see items 1. and 3. of
 Remark \ref{R1}.
\end{enumerate}
 \end{proof}

\end{rem}

\noindent In  Propositions \ref{P1} and \ref{P2} 
below we give two equivalent formulations for uniqueness of PDE 
 \eqref{EDPTerm0}. 
\begin{prop} \label{P1} 
 Suppose Property \ref{GH1} holds with respect to a given
 $\shc \subseteq \shm_+(\R^d)$.
  Suppose that for all $\nu \in \shc$
  there exists an $\shm_+(\R^d)$-valued solution of \eqref{Fokker}
  with initial value $\nu$.
 Then, the following properties are equivalent.
\begin{enumerate}
\item The mapping from $\mathcal{C}$ to  $\shm_+(\R^d)$
$\nu \mapsto {\bf u}^{\nu}(T)$
is injective.
\item For all $\mu \in  \shm_+(\R^d)$,
  the PDE \eqref{EDPTerm0}
 with terminal value $\mu$ admits at most a solution in the sense of Definition \ref{Def}, among all $\shm_+\left(\R^d\right)$-valued solutions starting in the class $\mathcal{C}$. 
\end{enumerate}
\end{prop}
\begin{proof}
\smallbreak
\noindent Concerning the converse implication,
suppose that uniqueness holds for equation 
\eqref{EDPTerm0}
in the sense
of Definition \ref{Def},
among non-negative measure-valued solutions starting in the class $\mathcal{C}$.
Consider $\nu, \nu'  \in \mathcal{C}$ such that 
${\bf u}^{\nu}(T) =  {\bf u}^{\nu'}(T).$
We remark that ${\bf u}^{\nu},{\bf u}^{\nu'}$ are such solutions
of PDE \eqref{EDPTerm0} with same 
terminal condition. Uniqueness gives ${\bf u}^{\nu} = {\bf u}^{\nu'}$ and in particular $\nu = \nu'$
and the injectivity stated in item 1. holds.
\smallbreak
 Concerning the direct implication, consider ${\bf u}^1,{\bf u}^2$
two non-negative measure-valued solutions of equation \eqref{EDPTerm0} in the sense of Definition \ref{Def}, with the same terminal value in $\shm_+\left(\R^d\right)$,
such that ${\bf u}^i\left(0\right), i \in \left\{1,2\right\},$ belong to $\mathcal{C}$ and suppose that $\nu \mapsto {\bf u}^{\nu}\left(T\right)$ is injective from $\mathcal{C}$ to $\shm_+\left(\R^d\right)$. Setting $\nu^i := {\bf u}^i\left(0\right)$, we remark that for a given $i \in \left\{1,2\right\}$ we have
\begin{equation} \label{FPBis}
\begin{cases}
\partial_t{\bf u}^i = L^*_t{\bf u}^i \\
{\bf u}^i\left(0\right) = \nu_i, \\
\end{cases}
\end{equation}
 in the sense of equation \eqref{weakbis}. 
Then, the fact ${\bf u}^1\left(T\right) = {\bf u}^2\left(T\right)$ gives 
$
{\bf u}^{\nu_1}\left(T\right) = {\bf u}^{\nu_2}\left(T\right).
$
By injectivity $\nu_1 = \nu_2$ and the statement 2. follows by Property \ref{GH1}.
\end{proof}
Proceeding in the same way as for the proof of Proposition \ref{P1},
for the case of signed measures, 
we obtain the following. 
\begin{prop}\label{P2}
  Suppose that for all $\nu \in \shm_f\left(\R^d\right)$,
  there exists a unique solution 
${\bf u}^\nu: [0,T] \rightarrow \shm_f\left(\mathbb{R}^d\right)$
  of the PDE \eqref{Fokker}
  with initial value $\nu$.
 Then, the following properties are equivalent.
\begin{enumerate}
\item The function 
$\nu \mapsto {\bf u}^{\nu}(T)$
is injective.
\item For all $\mu \in  \shm_f(\R^d)$,
  the PDE \eqref{EDPTerm0} with terminal value $\mu$ admits at most a solution in the sense of Definition \ref{Def}.
\end{enumerate}
\end{prop}

\begin{rem} \label{RP1}
\begin{enumerate}
\item Suppose that the coefficients $\Sigma, b$ are bounded.
 Then, any measure-valued solution ${\bf u}:[0,T] \rightarrow \shm_+(\R^d)$
  of the PDE \eqref{Fokker} such that  ${\bf u}(0) \in \shp(\R^d)$ takes values in $\shp(\R^d)$.
  Indeed, this can be shown approaching the function $\varphi \equiv 1$
  from below by smooth functions with compact support.
\item Replacing $\shm_+(\R^d)$ with $\shp(\R^d)$ in
 Property \ref{GH1}, 
  item 2. in Proposition \ref{P1} can be stated also replacing
  $\shm_+(\R^d)$ with $\shp(\R^d)$.
\end{enumerate}
  \end{rem}

\subsection{Uniqueness: the case of Dirac initial conditions}

\label{S32}

\noindent In this section we  will make use of a probabilistic technique
for discussing uniqueness of the PDE \eqref{EDPTerm0} among $\shm_+(\R^d)$-valued solutions
 starting in 
 $\mathcal{C} := \left(\alpha\delta_x\right)_{\alpha > 0,x \in \R^d}$.
We will make use of a probabilistic technique.
Given a solution ${\bf u}$ of \eqref{EDPTerm0}, we associate
a process $X$ 
being a solution of the SDE \eqref{eq:X} whose (marginal) law is ${\bf u(t)}$. 
The idea consists in identifying uniquely the law of $X_0$.
That approach only works with multiple Dirac initial conditions.

\smallbreak

\begin{rem} \label{Ralpha}
Let $\alpha \ge 0$ and $x \in \R^d$. 
 Suppose that there is a solution $X^x$ of SDE \eqref{EqLin} with  
$\xi = x$.
 \begin{enumerate}
\item  By Proposition \ref{PFundam},
  the $\shm_+\left(\R^d\right)$-valued mapping 
$t \mapsto \alpha\mathcal{L}\left(X^x_t\right)$ 
is a solution of the PDE \eqref{Fokker} with initial value $\alpha\delta_x$.
\item
   Under Property \ref{GH1} (with respect to $\shc$),
  $t \mapsto \alpha\mathcal{L}\left(X^x_t\right)$ can be identified with ${\bf u}^{\alpha\delta_x}$
  and in particular
  $$ \int_{\R^d} {\bf u}^{\alpha\delta_x}\left(t\right)\left(dy\right) = \alpha, \
  \forall t \in [0,T].$$
\end{enumerate}
\end{rem}
In the sequel,
 whenever Assumption \ref{Lip1d} holds, $X^x$ denotes the unique solution of 
the SDE
 \eqref{EqLin} with initial value $x \in \R^d$.

 We start with the case of dimension $d = m = 1$.
 
\begin{prop} \label{propLip1}
({\bf Uniqueness:  Dirac initial conditions, one-dimensional case}).
 
We set $\mathcal{C} = \left(\alpha\delta_x\right)_{\alpha > 0,x \in \R}$.
  Suppose the
  validity of Assumption \ref{Lip1d} with $d = m = 1$.
  We moreover suppose the validity of one of the two hypotheses below.
  \begin{enumerate}
  \item Assumption \ref{Zvon1a}.
  \item  Property \ref{GH1} holds with respect to $\shc$.
\end{enumerate}
Then,    
  for all $\mu \in \shm_+\left(\R\right)$, the PDE \eqref{EDPTerm0} with terminal value $\mu$ admits at most one solution in the sense of Definition \ref{Def} among the
  $\shm_+\left(\R\right)$-valued solutions starting in $\shc$.
\end{prop}

\begin{proof}  \
  \noindent 
By Lemma \ref{LC313} item 2.  Property \ref{GH1} is fulfilled
with respect to $\shc$.

  \noindent Fix $\left(x,y\right) \in \R^2$ and $\alpha,\beta \geq 0$ such that 
\begin{equation}\label{identity}
{\bf u}^{\alpha\delta_x}\left(T\right) = {\bf u}^{\beta\delta_y}\left(T\right).
\end{equation}
 Thanks to Proposition \ref{P1}, to conclude,  
it suffices to show that $\alpha = \beta$ and $x=y$.
 By item 2. of Remark \ref{Ralpha}, we have $\alpha = \beta$ and consequently $\mathcal{L}\left(X^x_T\right) = \mathcal{L}\left(X^y_T\right)$.
In particular $\mathbb{E}\left(X^x_T\right)  = \mathbb{E}\left(X^y_T\right)$. Since $b,\sigma$ are Lipschitz in space, they have bounded derivatives in the sense of distributions that we denote by $\partial_xb$ and $\partial_x\sigma$. 
\smallbreak
\noindent Set $Z^{x,y} := X^y - X^x$. We have
\begin{equation} \label{EDol}
Z^{x,y}_t = \left(y-x\right) + \int^{t}_{0}b^{x,y}_sZ^{x,y}_sds + \int^{t}_{0}\sigma^{x,y}_sZ^{x,y}_sdW_s, \forall
t\in [0,T],
 \end{equation}
\noindent where for a given $s \in [0,T]$
\be
b^{x,y}_s = \int^{1}_{0}\partial_xb\left(s,aX^y_s + (1-a)X^x_s\right)da ,\ \sigma^{x,y}_s = \int^{1}_{0}\partial_x\sigma\left(s,aX^y_s + (1-a)X^x_s\right)da.
\ee
\smallbreak
\noindent 
The unique solution of \eqref{EDol} is well-known to be
\be 
Z^{x,y} = \exp\left(\int^{.}_{0}b^{x,y}_sds\right)\mathcal{E}\left(\int^{.}_{0}\sigma^{x,y}_sdW_s \right)(y-x),
\ee 
 where $\mathcal{E}\left(\cdot\right)$ denotes the Dol\'eans exponential.
 Finally, we have 
\be 
\mathbb{E}\left(\exp\left(\int^{T}_{0}b^{x,y}_sds\right)\mathcal{E}\left(\int^{.}_{0}\sigma^{x,y}_sdW_s \right)_T\right)\left(y-x\right) = 0.
\ee 
 Since  the quantity appearing in the expectation is strictly positive, we conclude $x = y$.
\end{proof}

\noindent We continue now with a discussion concerning the multidimensional case $d \geq 2$.
The  uniqueness result below only holds when the time-horizon is small enough.
Theorem \ref{propLipd} distinguishes two cases: the first one with regular,
 possibly
 degenerate, coefficients, the second one with non-degenerate, possibly irregular, coefficients. Later, in Section \ref{SGP},
 we will present in a framework of piecewise time-homogeneous coefficients
results which are valid  for any time-horizon.
\begin{thm} \label{propLipd}
({\bf Uniqueness: Dirac initial conditions, multi-dimensional case}). 

We set $\mathcal{C} = \left(\alpha\delta_x\right)_{\alpha > 0,x \in \R^d}$.
  We suppose the validity of either item (a)
 or (b) below.
  \begin{description}
\item{(a)} Assumptions
  \ref{Lip1d}
    and Property \ref{GH1} (for instance if Assumption \ref{Zvon1a}
holds)  with respect to $\mathcal{C}$.
\item{(b)} Assumptions \ref{Zvon1a}, \ref{Zvon1b}
   and \ref{Zvon3}.
\end{description}
There is $T > 0$ small enough such that the following holds.
 For all $\mu \in \shm_+\left(\R^d\right)$, the PDE \eqref{EDPTerm0}
admits at most one solution in the sense of Definition \ref{Def} among the $\shm_+\left(\R^d\right)$-valued solutions starting in $\shc$.
\end{thm}
 The proof of Theorem \ref{propLipd} in case (a) 
relies on a basic lemma of moments estimates.

\begin{lem} \label{Lemma} We suppose Assumption \ref{Lip1d}. 
Let $\left(x,y\right) \in \R^d\times\R^d$. Then,
 $ \sup_{t\in[0,T]}\mathbb{E}
\left(\left|X^x_t - X^y_t \right|^2 \right)  \le    \left|y-x\right|^2e^{KT},$
 with $K := 2K^b + \sum^{d}_{j=1}\left(K^{\sigma,j}\right)^2$, where 
\be
K^b := \sup_{s\in [0,T]}\left|\left|   \  \left|\left|Jb\left(s,\cdot\right)\right|\right|  \   \right|\right|_{\infty}
\ee
 and for all $j \in [\![1,d]\!]$
\be 
K^{\sigma,j} := \sup_{s\in [0,T]}\left|\left|   \   \left|\left|J\sigma_{.j}\left(s,\cdot\right)\right|\right| \   \right|\right|_{\infty},
\ee
where $\Vert \cdot \Vert$ stands for the sup-norm.
\end{lem}
\begin{prooff} \ (of Theorem  \ref{propLipd}).


  Taking into account Property \ref{GH1}
 we fix $\left(x_1,x_2\right) \in \R^d\times\R^d, \alpha,\beta \geq 0$
  such that 
\begin{equation}
{\bf u}^{\alpha\delta_{x_1}}\left(T\right) = {\bf u}^{\beta\delta_{x_2}}\left(T\right).
\end{equation}
To conclude, by  Proposition \ref{P1},
it suffices again to show $\alpha = \beta$ and $x_1 = x_2$.
\smallbreak
\begin{enumerate}
\item We write the proof in the case (a), in particular
  under Assumption \ref{Lip1d}.
  Once again, item 2. of Remark \ref{Ralpha}
  gives $\alpha = \beta$ and
\begin{equation} \label{Eequal}
  \mathbb{E}\left(X^{x_1}_T\right) = \mathbb{E}\left(X^{x_2}_T\right).
  \end{equation}
\smallbreak
\noindent Adopting the same notations as in the proof of Lemma \ref{Lemma},
a similar argument as in \eqref{EMForm}, together with \eqref{sup}
(in the Appendix) allows to show that 
the local martingale part of $Z^{x_1,x_2} = X^{x_2} - X^{x_1}$ defined in
\eqref{EZxy} is a true martingale. So, taking the expectation
in  \eqref{EMForm} with $x=x_1, y = x_2$,
 by Lemma \ref{Lemma} we obtain
 \begin{align*}
   \left|\mathbb{E} \left(X^{x_2}_T - X^{x_1}_T\right) - (x_2-x_1)\right| &{}\le
    K_b \int^{T}_{0}\mathbb{E}  \vert X^{x_2}_r - X^{x_1}_r \vert dr \\
              &{}\leq  K_b \int^{T}_{0}
     \sqrt{\mathbb{E}\left(\vert X^{x_2}_r - X^{x_1}_r \vert\right)^2}dr \\
&{}\leq \frac{K}{2} T e^{\frac{K}{2}T} \left|x_2-x_1\right|. 
\end{align*}
 Remembering \eqref{Eequal},
this implies
\be
\left(1 - \frac{K}{2}Te^{\frac{K}{2}T}\right)\left|x_2-x_1\right| \leq 0.
\ee 
Taking $T$ such that $\frac{K}{2}T < M$ with $Me^M < 1$, we have
$1 - \frac{K}{2}Te^{\frac{K}{2}T} > 0$, which implies $ \vert x_2 - x_1 \vert=  0$.
\item We discuss the case (b), i.e. we suppose Assumptions \ref{Zvon1a},  \ref{Zvon1b},
 and \ref{Zvon3}.
Firstly, point 1. of Theorem 1. in \cite{z} ensures the existence of probability spaces $\left(\Omega^i, \mathcal{F}^i,\mathbb{P}^i\right), \ i \in \left\{1,2\right\}$ on which are defined respectively two $m$-dimensional Brownian motions $W^1,W^2$ and two processes $X^1,X^2$ such that
\be
X^i_t = x_i + \int^{t}_{0}b\left(s,X^i_s\right)ds + \int^{t}_{0}\sigma\left(s,X^i_s\right)dW^i_s, \ \mathbb{P}^i\rm{-a.s.}, t \in [0,T].
\ee
\smallbreak

\noindent Again item 2. of Remark \ref{Ralpha} implies $\alpha_1 = \alpha_2$ and
\begin{equation} \label{TermLaw}
\mathcal{L}_{\mathbb{P}^1}\left(X^1_T\right) =  \mathcal{L}_{\mathbb{P}^2}\left(X^2_T\right).
\end{equation}
 Secondly, point b. of Theorem 3 in \cite{z} shows that 
for every given bounded $D \subset \R^d$, for all
$\phi: [0,T] \times \R^d \rightarrow \R^d$ belonging to
$W^{1,2}_p\left([0,T]\times D\right)$ (see Definition of that space in \cite{z})
for a given $p > d+2$, for all $t\in [0,T], i \in \left\{1,2\right\}$, we have
\begin{equation}\label{TSDE}
\phi\left(t,X^i_t\right) = \phi\left(0,x_i\right) + \int^{t}_{0}\left(\partial_t + L_s\right)\phi\left(s,X^i_s\right)ds + \int^{t}_{0}J\phi\left(s,X^i_s\right)\sigma\left(s,X^i_s\right)dW^i_s, \ \mathbb{P}^i\rm{-a.s.}
\end{equation}
 where the application of $\partial_t + L_t, t \in [0,T]$ has to be understood componentwise.
\smallbreak 
\noindent Thirdly, Theorem 2. in \cite{z} shows that if $T$ is sufficiently small, then the system of $d$ PDEs
\begin{equation}\label{E317}
\forall \left(t,x\right) \in [0,T]\times\R^d, \
\begin{cases}
\partial_t\phi\left(t,x\right) + L_t \phi\left(t,x\right) = 0,  \\
\phi\left(T, x\right) = x,
\end{cases}
\end{equation}
 admits a solution $\phi$ in $W^{1,2}_p\left([0,T]\times D\right)$ for all $p > 1$ and all
bounded $D \subset \R^d$.
Moreover the partial derivatives of $\phi$ in space are bounded (in particular $J \phi$ is bounded) and 
 $\phi\left(t,\cdot\right)$ is injective for all $t \in [0,T]$. 
\smallbreak
\noindent Combining now \eqref{E317}  with identity \eqref{TSDE}, we observe that $\phi\left(.,X^i\right), i \in \left\{1,2\right\},$ are local martingales. Using additionally the fact
that $J\phi$ and $\sigma$ are bounded,
it is easy to show that they are true martingales. Taking the expectation
in \eqref{TSDE} with respect to $\P^i, i =1,2$,
 gives 
\be
\phi\left(0,x_i\right) =  \mathbb{E}_{\mathbb{P}^i}\left(\phi\left(T,X^i_T\right)\right),  i \in \left\{1,2\right\}.
\ee 
 In parallel, identity \eqref{TermLaw} gives 
\be
\mathbb{E}_{\mathbb{P}^1}\left(\phi\left(T,X^1_T\right)\right) =  \mathbb{E}_{\mathbb{P}^2}\left(\phi\left(T,X^2_T\right)\right).
\ee
 So, $\phi\left(0,x_1\right) =\phi\left(0,x_2\right)$. We conclude
that $x_1 = x_2$ since  $\phi\left(0,\cdot\right)$ is injective.
\end{enumerate}
\end{prooff}

\subsection{Uniqueness: the case of bounded non-degenerate
  coefficients}
\label {SGP}

In this section we consider the case of
 (possibly piecewise) time-homogeneous coefficients
in dimension $d \geq 1$.
We make use of an analytic technique based on semigroups which requires
bounded coefficients (Assumption \ref{Zvon1a}), non-degeneracy 
(Assumption \ref{Zvon3}) and an additional H\"older regularity assumption
of the coefficients.
  
We start with the time-homogeneouse case stating the following.

\begin{ass}\label{Lun1}
\begin{enumerate}
\item $b,\Sigma$ are time-homogeneous.
\item
For all $\left(i,j\right) \in [\![1,d]\!]^2$,
$b_i, \Sigma_{ij} \in \mathcal{C}^{2\alpha}\left(\R^d\right)$,
 for a given $\alpha \in ]0,\frac{1}{2}[$.
\end{enumerate}
\end{ass}
We refer  to the differential operator $L_t$ introduced in \eqref{EqOpL} 
and
we simply set here $L \equiv L_t$.
\begin{rem} \label{RPreliminary}
  Suppose the validity of Assumptions \ref{Zvon1a},
 \ref{Zvon3}, \ref{Lun1}.
\begin{enumerate} 
\item Let $T > 0$. Proposition 4.2 in \cite{figalli} implies that, for every
 $\nu \in \shm_f\left(\R^d\right)$, there exists a unique 
$\shm_f\left(\R^d\right)$-valued solution of the PDE \eqref{Fokker} with initial value $\nu$,
  which will be again denoted by ${\bf u}^{\nu}$.
  We notice in particular that Property \ref{GH1} holds.

 In the sequel $T$ will be omitted. 
\item We remark that the uniqueness result mentioned in item 1.
is unknown in the case of general bounded coefficients.
In the general framework, only a uniqueness result for non-negative 
solutions is available, see Remark \ref{R1} point 1.
    \item
   Since $L$ 
is time-homogeneous, taking into account
   Assumptions \ref{Zvon3}, \ref{Lun1}, 
   operating a shift,
   uniqueness for the PDE \eqref{Fokker} 
   also holds replacing the initial time $0$ by any other initial time,
   for every initial value in $\shm_f\left(\R^d\right)$, 
   with any other
maturity $T$.
\end{enumerate}
\end{rem}
It is significant to remark that the uniqueness theorem
below holds in the class  finite signed measures valued functions.
\begin{thm} \label{P315}
({\bf Uniqueness: the case of non-degenerate time-homogeneous 
coefficients}).

Suppose the validity of Assumptions \ref{Zvon1a}, 
\ref{Zvon3} and \ref{Lun1}. 
Then, for all $\mu \in \shm_f\left(\R^d\right)$, the PDE \eqref{EDPTerm0} 
with terminal value $\mu$ admits at most one $\shm_f\left(\R^d\right)$-valued solution in the sense of Definition \ref{Def}.
\end{thm}
 
  By Theorems 3.1.12, 3.1.14 and Corollary 3.1.16 in \cite{lunardi_1995} 
the differential operator $L$ suitably extends as a map 
 $\shd(L) = \mathcal{C}^{2\alpha+2}(\R^d) \subset \mathcal{C}^{2\alpha}(\R^d) 
 \mapsto  \mathcal{C}^{2\alpha}\left(\R^d\right)$ 
and that extension 
is sectorial, see  Definition 2.0.1 in \cite{lunardi_1995}.
We set $E:= \mathcal{C}^{2\alpha}\left(\R^d\right)$. 
By the considerations below that Definition, in (2.0.2) and (2.0.3)
therein, one defines
$P_t := e^{tL}, P_t: E \rightarrow E, t \geq 0$.
 By Proposition 2.1.1 in \cite{lunardi_1995},
$(P_t)_{t \geq 0}$ is a semigroup and $t \mapsto P_t$ is analytical
on $]0,+\infty[$ with values in $\mathcal{L}\left(E\right)$, with respect to
$\left|\left|.\right|\right|_{E}$. 
\smallbreak

\noindent Before proving the theorem, we provide two lemmata.
\begin{lem} \label{key_1}
\noindent Suppose the validity of Assumptions \ref{Zvon1a},
 \ref{Zvon3} and \ref{Lun1}. 
 Then, for all $\phi \in E$
 and all $\nu \in \shm_f\left(\R^d\right)$, the function from $\R^*_+$ to $\mathbb{R}$
\be
t \mapsto \int_{\R^d}P_t\phi\left(x\right)\nu\left(dx\right) 
\ee
 is analytic.
\end{lem}

\begin{proof} \
The result can be easily established 
using the fact that $\phi \mapsto P_t \phi$ with values in $\shl(E)$
is analytic and the fact that the map $\psi \mapsto \int_{\R^d} \psi(x) \nu(dx)$ 
is linear and bounded.

\end{proof}
\begin{lem} \label{key_2}
Suppose the validity of Assumptions \ref{Zvon1a}, \ref{Zvon3} and \ref{Lun1}. Let $T > 0$.
Then for all $\nu \in \shm_f\left(\R^d\right)$, $t \in [0,T]$ and 
$\phi \in E$
we have the identity
\begin{equation} \label{EL310}
\int_{\R^d}P_{t}\phi\left(x\right)\nu\left(dx\right) = \int_{\R^d}\phi\left(x\right){\bf u}^{\nu}\left(t\right)\left(dx\right),
\end{equation}
where ${\bf u}^{\nu}$ was defined in point 1. of Remark \ref{RPreliminary}.
\end{lem}
\begin{proof}
\
\noindent 
Let $\nu \in \shm_f\left(\R^d\right)$.
 We denote by ${\bf v}^{\nu}$ the mapping from $[0,T]$ to 
$\mathcal{M}_f\left(\R^d\right)$ such that $\forall t \in [0,T]$,
 $\forall \phi \in E$
\begin{equation} \label{ERiesz}
\int_{\R^d}\phi(x){\bf v}^{\nu}\left(t\right)\left(dx\right)= \int_{\R^d}P_{t}\phi(x)\nu (dx).
\end{equation}
Previous expression defines the measure ${\bf v}^{\nu}(t,\cdot)$ since
$\phi \mapsto  \int_{\R^d}P_{t}\phi(x)\nu (dx)$ is continuous with respect 
to the sup-norm, using $\Vert P_t\phi \Vert_\infty \le \Vert \phi \Vert_\infty$,
and Lebesgue's dominated convergence theorem.
 By approximating the elements of $E$ with 
elements of $\mathcal{C}^{\infty}_c\left(\R^d\right)$,
 it will be enough to prove  
\eqref{EL310} for $\phi \in \mathcal{C}^{\infty}_c\left(\R^d\right)$.

Our idea is to show that ${\bf v}^{\nu}$ is an $\shm_f\left(\R^d\right)$-valued solution of \eqref{Fokker} with initial value $\nu$, so that
${\bf v}^{\nu} = {\bf u}^{\nu}$ via point 1. of Remark \ref{RPreliminary}.
This will prove \eqref{EL310} for $\phi \in \mathcal{C}^{\infty}_c\left(\R^d\right)$.
Let $t \in [0,T]$ and
$\phi \in \mathcal{C}^{\infty}_c\left(\R^d\right)$. On the one hand,
point (i) of Proposition 2.1.1 in \cite{lunardi_1995} gives
\begin{equation} \label{LP}
LP_t\phi = P_tL\phi, 
\end{equation}
 since $\mathcal{C}^{\infty}_c\left(\R^d\right) \subset \mathcal{D}\left(L\right) = \mathcal{C}^{2\alpha + 2}\left(\R^d,\R\right)$. On the other hand, for all $s \in [0,t]$, we have 
\begin{align*}
\left|LP_s\phi\right|_{E} &{}= \left|P_sL\phi\right|_{2\alpha} \\
&{}\leq \left|\left|P_s\right|\right|_{E}\left|L\phi\right|_{E} \\
&{}\leq M_0e^{\omega s}\left|L\phi\right|_{E}, 
\end{align*}
 with $M_0,\omega$ the real parameters appearing in Definition 2.0.1 in \cite{lunardi_1995} and using point (iii) of Proposition 2.1.1 in the same reference.
 Then the mapping $s \mapsto LP_s\phi$ belongs obviously to
 $L^1([0,t];E)$ and point (ii) of Proposition 2.1.4 in \cite{lunardi_1995} combined with identity \eqref{LP} gives 
\be
P_t\phi = \phi + \int^{t}_{0}P_s L\phi ds.
\ee
 Back to our main goal, using in particular Fubini's theorem, we have 
\begin{align*}
\int_{\R^d}P_{t}\phi\left(x\right)\nu\left(dx\right) &{} = \int_{\R^d}\phi\left(x\right)\nu\left(dx\right) + \int_{\R^d}\int^{t}_{0}P_sL\phi\left(x\right)ds\nu\left(dx\right) \\
&{}= \int_{\R^d}\phi\left(x\right)\nu\left(dx\right) + \int^{t}_{0}\int_{\R^d}P_sL\phi\left(x\right)\nu\left(dx\right)ds  \\
&{}= \int_{\R^d}\phi\left(x\right)\nu\left(dx\right) + \int^{t}_{0}\int_{\R^d}L\phi\left(x\right){\bf v}^{\nu}\left(s\right)\left(dx\right)ds. 
\end{align*}
 This shows that ${\bf v}^{\nu}$ is a solution of
 the PDE \eqref{Fokker}.
\end{proof}

\begin{prooff} ((of Theorem \ref{P315}).

\noindent Let $\nu,\nu' \in \shm_f\left(\R^d\right)$ such that
\be
\mu_T : = {\bf u}^{\nu}\left(T\right) = {\bf u}^{\nu'}\left(T\right).
\ee
Thanks to Proposition \ref{P2}, it suffices to show that
$\nu = \nu'$ i.e.
 \be
\forall \phi \in \mathcal{C}^{\infty}_c\left(\R^d\right),\ \int_{\R^d}\phi\left(x\right)\nu\left(dx\right) = \int_{\R^d}\phi\left(x\right)\nu'\left(dx\right).
\ee
 Since $T > 0$ is arbitrary, by Remark \ref{RPreliminary}
we can consider ${\bf u}^{\nu,2T}$ and ${\bf u}^{\nu',2T}$,
defined as the corresponding  ${\bf u}^{\nu}$ and ${\bf u}^{\nu'}$
functions obtained replacing the horizon $T$ with $2T$.
They are defined 
on $[0,2T]$  and
by Remark \ref{RPreliminary} 1.
(uniqueness on $[0,T]$),
 they constitute extensions of the initial ${\bf u}^{\nu}$ and ${\bf u}^{\nu'}$.

\noindent By Remark \ref{RPreliminary} 3., the
 uniqueness of an $\shm_f\left(\R^d\right)$-valued solution of the PDE
\eqref{Fokker}
(for $t \in  [T,2T]$, with $T$ as initial time)
holds for   
\begin{equation}\label{FPShift}
\begin{cases}
\partial_t {\bf u}(\tau) = L^*{\bf u}(\tau), \  T \leq \tau \leq 2T \\
{\bf u}(T) = \mu_T. \\
\end{cases}
\end{equation}
Now, the functions ${\bf u}^{\nu, 2T}$ and ${\bf u}^{\nu', 2T}$ solve
\eqref{FPShift} on $[T,2T]$.
This gives in particular 
\begin{equation} \label{IdLawBis}
\forall \tau \geq T, \ \forall \phi \in \mathcal{C}^{\infty}_c\left(\R^d\right), \ \int_{\R^d}\phi\left(x\right){\bf u}^{\nu,2T}\left(\tau\right)\left(dx\right) = \int_{\R^d}\phi\left(x\right){\bf u}^{\nu',2T}\left(\tau\right)\left(dx\right).
\end{equation}
 Fix $\phi \in \mathcal{C}^{\infty}_c\left(\R^d\right)$. Combining
now the results of Lemmata \ref{key_1} 
and \ref{key_2}, 
we obtain that the function
\begin{equation} \label{ETau}
\tau \mapsto  \int_{\R^d}\phi\left(x\right){\bf u}^{\nu,2T}\left(\tau\right)\left(dx\right) - \int_{\R^d}\phi\left(x\right){\bf u}^{\nu',2T}\left(\tau\right)\left(dx\right),
\end{equation}
\noindent defined on $[0, 2T]$, is zero on $[T,2T]$ and analytic on 
$]0,2T]$. Hence it is zero on $]0,2T]$.
By \eqref{EL310} we obtain
\begin{equation} \label{ETaubis}
  \int_{\R^d}P_{\tau}\phi(x) \left(\nu - \nu'\right)
  \left(dx\right) = 0, \ \forall  \tau \in ]0, 2T].
\end{equation}
Separating $\nu$  and $\nu'$ in positive and negative components,
we can finally apply
dominated convergence theorem in \eqref{ETau} to send $\tau$ to $0+$.
This is possible
thanks to points
 (i) of Proposition 2.1.4 and (iii) of Proposition 2.1.1 in \cite{lunardi_1995}
together with the representation \eqref{EL310}. 
Indeed $P_\tau\phi\left(x\right) \rightarrow \phi\left(x\right)$  for every $\phi \in E, x \in \R^d$ when $\tau \rightarrow 0+$.
This shows $\nu = \nu'$ and ends the proof.

\end{prooff}
 For the sake of applications it is useful to formulate 
a piecewise time-homogeneous version of Theorem \ref{P315}.

\begin{corro} \label{C313}
({\bf Uniqueness: the case of non-degenerate piecewise time-homogeneous 
coefficients}).

 Let $n \in \N^*$.
Let 
$ 0 = t_0 < \ldots < t_n = T$ be a partition. 
For  $k \in [\![2,n]\!]$ (resp. $k=1$) we
denote $I_k = ]t_{k-1},t_k]$ (resp. $[t_{0},t_1]$).
Suppose that
the following holds.
\begin{enumerate}
\item For all $k \in [\![1,n]\!]$, the restriction of $\sigma$ (resp. $b$)  to $I_k \times \R^d$ 
is a time-homogeneous function $\sigma^k: \R^d \rightarrow M_{d}(\R)$
 (resp. $b^k: \R^d \rightarrow \R^d$).
\item Assumption \ref{Zvon3}.
\item Assumptions \ref{Zvon1a}
 and \ref{Lun1} are verified for each
$\sigma^k, b^k$ and $\Sigma^k$, where we have
set $\Sigma^k := \sigma^k {\sigma^k}^\top$.

\end{enumerate}
 Then, for all $\mu \in \shm_f\left(\R^d\right)$, the PDE
\eqref{EDPTerm0} with terminal value $\mu$ admits at most one $\shm_f\left(\R^d\right)$-valued solution in the sense of Definition \ref{Def}.
\end{corro}
\begin{proof}
\noindent For each given $k \in [\![1,n]\!]$, we introduce the PDE 
operator $L^k$ defined by 
\begin{equation} \label{OpLk}
L^k := \frac{1}{2}\sum^{d}_{i,j=1}\Sigma^k_{ij}\partial_{ij} +  
\sum^{d}_{i=1}b^k_i\partial_{i}.
\end{equation}
 Let now ${\bf u}^1, {\bf u}^2$ be two solutions  of \eqref{EDPTerm0} with same terminal value $\mu$.

\smallbreak
\noindent The measure-valued functions ${\bf v}^i := {\bf u}^i\left(\cdot + t_{n-1}\right), i \in \{1,2\}$ defined on $[0,T-t_{n-1}]$ are solutions of 
\begin{equation} \label{BackwardFokker_k}
\begin{cases}
\partial_t{\bf v} = \left(L^n\right)^*{\bf v} \\
{\bf v}\left(T-t_{n-1},\cdot\right) = \mu,
\end{cases} 
\end{equation}
\noindent in the sense of Definition \ref{Def} replacing $T$ by $T-t_{n-1}$ and $L$ by $L^n$. Then, Theorem \ref{P315} gives ${\bf v}^1 = {\bf v}^2$ and consequently ${\bf u}^1 = {\bf u}^2$ on $[t_{n-1},T]$. To conclude, we proceed by backward induction.

\end{proof}

\subsection{Uniqueness: the case of Ornstein-Uhlenbeck
 semigroup}

\label{S34}

\smallbreak

\noindent In this section, we consider the case $b := \left(s,x\right) \mapsto C(s)x$ with $C$ continuous from $[0,T]$ to $M_d\left(\R\right)$ and $\sigma$ continuous from $[0,T]$ to $M_d\left(\R\right)$.
Here we perform an analytic approach based on Fourier analysis.

 We recall that $\Sigma := \sigma \sigma^{\top}$.
 In that setting, the classical Fokker-Planck PDE \eqref{EDPInitial} for finite measures reads
\begin{equation} \label{FP_OU}
\begin{cases} 
\displaystyle \partial_t {\bf u}\left(t\right) = \frac{1}{2}  \sum^{d}_{i,j=1}\Sigma(t)_{ij}\partial_{ij}{\bf u}\left(t\right) - \sum^{d}_{i=1}\partial_i\left(\left(C(t)x\right)_i{\bf u}\left(t\right)\right) \\
{\bf u}(0) = \nu \in \shm_f\left(\R^d\right).
\end{cases}
\end{equation}
In the sequel we will denote by $\shd\left(t\right), \ t \in[0,T],$ the unique solution
of
\begin{equation} \label{ED}
  \shd(t) = I - \int^t_0 C(s)^{\top}\shd(s)ds, \ t \in [0,T].
  \end{equation}
\smallbreak
\noindent We recall that for every $t \in [0,T]$, $\shd(t)$ is invertible
and 
\begin{equation} \label{ED-1}
  \shd^{-1}(t) = I + \int^t_0 \shd^{-1}(s)  C(s)^{\top}ds,  \ t \in [0,T].
  \end{equation}
For previous and similar properties, see Chapter 8 of \cite{bronson}.

\begin{prop} \label{FwdOU_Uniq}
 For all $\nu \in \shm_f\left(\R^d\right)$, the PDE
\eqref{FP_OU}
with initial value $\nu$ admits at most one $\shm_f\left(\R^d\right)$-valued solution. In particular Property \ref{GH1} holds for
$ \shc = \shm_+\left(\R^d\right)$.
\end{prop}
\begin{proof} \
\begin{enumerate}
 \item Let $\nu \in \shm_f\left(\R^d\right)$ and ${\bf u}$ be a solution of 
the PDE \eqref{Fokker} with initial value $\nu$. 
Identity \eqref{weakbis} can be extended to $\shs\left(\R^d\right)$ since 
for all $t \in [0,T]$, ${\bf u}\left(t\right)$ belongs to 
$\shm_f\left(\R^d\right)$.
Then, $t \mapsto \shf{\bf u}\left(t\right)$ verifies 
\begin{equation} \label{WeakOUPDE}
\shf{\bf u}\left(t\right)(\xi) = \shf \nu(\xi) + \int^{t}_{0}\left<C\left(s\right)^{\top}\xi,\nabla \shf {\bf u}\left(s\right)\right > ds - \frac{1}{2}\int^{t}_{0}\left<\Sigma\left(s\right)\xi,\xi\right>\shf{\bf u}\left(s\right)ds, \ (t,\xi) \in [0,T]\times\R^d.
\end{equation}
 In fact, the integrand inside the first integral 
has to be understood as a Schwartz distribution:
in particular
 the symbol $\nabla$ is understood in the sense of distributions 
and for each given $s \in [0,T]$, $\left<C\left(s\right)^{\top}\xi,\nabla \shf {\bf u}\left(s\right)\right>$ denotes the tempered distribution
\be
\varphi \mapsto \sum^{d}_{i=1}\partial_i\shf {\bf u}\left(s\right)\left(\xi
 \mapsto \left(C\left(s\right)^{\top}\xi\right)_i\varphi\left(\xi\right)\right).
\ee
Indeed, even though for any $t$, $\shf{\bf u}\left(t\right)$ is a function,
the equation  \eqref{WeakOUPDE} has to be understood in
$\shs'\left(\R^d\right)$.
Hence, for all $\phi \in \shs\left(\R^d\right)$, this gives
\begin{align}
  \int_{\R^d}\phi\left(\xi\right)\shf{\bf u}\left(t\right)\left(\xi\right)d\xi &{}- \int_{\R^d}\phi\left(\xi\right)\shf\nu\left(\xi\right)
\phi(\xi)
d\xi\\
  &{}=  - i \sum^{d}_{k,l=1}\int^{t}_{0} C\left(s\right)_{kl} \int_{\R^d}\xi_l\shf\phi_k\left(\xi\right){\bf u}\left(s\right)\left(d\xi\right) ds -
    \frac{1}{2}\int^{t}_{0}\int_{\R^d}\left<\Sigma\left(s\right)\xi,\xi\right>\shf{\bf u}\left(s\right)\left(\xi\right) \phi(\xi) d\xi ds
\nonumber
  \\
                                   &{}= - \sum^{d}_{k,l=1}\int^{t}_{0}    C\left(s\right)_{kl}   \int_{\R^d}
     \shf\left(\partial_l\phi_k\right)\left(\xi\right){\bf u}\left(s\right)
                                     \left(d\xi\right)ds - \frac{1}{2}\int^{t}_{0}\int_{\R^d}\left<\Sigma\left(s\right)\xi,\xi\right>\shf{\bf u}\left(s\right)\left(\xi\right)d\xi ds
                                     \nonumber \\
  &{}= - \int^{t}_{0}\int_{\R^d}\left(\mathop{div_\xi}\left(C\left(s\right)^{\top}\xi\phi\left(\xi\right)\right) +
    \frac{1}{2}\left<\Sigma\left(s\right)\xi,\xi\right>
    \phi\left(\xi\right)\right)\shf{\bf u}(s)(\xi) d\xi ds,
                                                                                 \nonumber
\end{align}
 where $\phi_k : \xi \mapsto \xi_k\phi\left(\xi\right)$ for a given $k \in [\![1,d]\!]$.

\smallbreak
\noindent \item Let now ${\bf v}: [0,T] \rightarrow
\shm_f\left(\R^d\right)$
be defined by 
\begin{equation} \label{MeasChange}
\int_{\R^d}\phi\left(x\right){\bf v}\left(t\right)\left(dx\right) = \int_{\R^d}\phi\left(\shd\left(t\right)^{\top}x\right){\bf u}\left(t\right)\left(dx\right),
\end{equation}
 $t \in [0,T], \phi \in \shc_b(\R^d).$
For every $\xi \in \R^d$, we set $\phi(x) = \exp(-i \langle \xi, x\rangle)$
in \eqref{MeasChange} to obtain
\begin{equation} \label{EAcont}
\shf{\bf v}\left(t\right)\left(\xi\right) =  \shf{\bf u}\left(t\right)\left(\shd\left(t\right)\xi\right),
\end{equation}
 for all $\xi \in \R^d$, for all $t \in [0,T]$.

\noindent \item
 We want now to show that, for each $\xi$, 
$t \mapsto \shf{\bf v}\left(t\right)$ fulfills an ODE.
To achieve this, suppose for a moment that $ \left(t,\xi\right) \mapsto \shf{\bf u}\left(t\right)\left(\xi\right)$ is differentiable with respect to the variable $\xi$. Then, on the one hand, for all $\left(t,\xi\right)\in[0,T]\times\R^d$,
we have
\begin{equation}\label{strongPDE}
\shf{\bf u}\left(t\right)\left(\xi\right) = \shf\nu\left(\xi\right) + \int^{t}_{0}\left<C\left(s\right)^{\top}\xi,\nabla_\xi\shf {\bf u}\left(s\right)\left(\xi\right)\right>ds - \frac{1}{2}\int^{t}_{0}\left<\Sigma\left(s\right)\xi,\xi\right>\shf{\bf u}\left(s\right)\left(\xi\right)ds,
\end{equation}
 thanks to identity \eqref{WeakOUPDE}. This means in particular that, for each given $\xi \in \R^d$, $t \mapsto \shf{\bf u}\left(t\right)\left(\xi\right)$ is differentiable almost everywhere on $[0,T]$.
\smallbreak

\noindent On the other hand, 
 for almost every $t \in [0,T]$ and all $\xi \in \R^d$, we have
\begin{align} \label{ETechnical}
\partial_t\shf {\bf v}\left(t\right)\left(\xi\right)&{}=\partial_t\shf {\bf u}\left(t\right)\left(\shd\left(t\right)\xi\right) + \sum^{d}_{i=1}\left(\frac{d}{dt}\left(\shd\left(t\right)\xi\right)\right)_i    \partial_i\shf {\bf u}\left(t\right)\left(\shd\left(t\right)\xi\right), \nonumber \\
&{}= \partial_t\shf {\bf u}\left(t\right)\left(\shd\left(t\right)\xi\right) - \sum^{d}_{i=1}\left(C\left(t\right)^{\top}\shd\left(t\right)\xi\right)_i\partial_i\shf {\bf u}\left(t\right)\left(\shd\left(t\right)\xi\right), \nonumber\\
&{}= - \frac{1}{2}\left<\Sigma\left(t\right)\shd\left(t\right)\xi,\shd\left(t\right)\xi\right>\shf{\bf v}\left(t\right)\left(\xi\right), 
\end{align}
where  from line 1 to line 2, 
   we have used the fact $\frac{d}{dt}\left(\shd\left(t\right)\xi\right) =
 - C\left(t\right)^{\top}\shd\left(t\right)\xi$ for all $\left(t,\xi\right) \in [0,T]\times\R^d$  and from line 2 to line 3, 
the identity \eqref{strongPDE}.
Since $t \mapsto  \shf {\bf v}\left(t\right)\left(\xi\right)$
is absolutely continuous by \eqref{EAcont}, \eqref{ETechnical}
  implies
\begin{equation}\label{EDOFourierFwd}
\shf {\bf v}\left(t\right)\left(\xi\right) = \shf \nu \left(\xi\right) - \frac{1}{2}\int^{t}_{0}\left<\Sigma\left(s\right)\shd\left(s\right)\xi,\shd\left(s\right)\xi\right>\shf {\bf v}\left(s\right)\left(\xi\right)ds,
 \xi \in \R^d,
\end{equation}
for all $t \in [0,T]$.

\noindent \item Now, if $\left(t,\xi\right) \mapsto \shf {\bf u}\left(t\right)\left(\xi\right)$ is not
necessarily differentiable in the variable $\xi$, 
we will be able to prove
\eqref{EDOFourierFwd} still holds by making use of calculus in the
sense of distributions.

\noindent \item Suppose that \eqref{EDOFourierFwd} holds.
This gives
\begin{equation}\label{FourierExplicitOU}
\shf{\bf u}\left(t\right)\left(\xi\right) = e^{-\int^{t}_{0}\frac{\left|\sigma\left(s\right)^{\top}\xi\right|^2}{2}ds}\shf \nu\left(\shd^{-1}\left(t\right)\xi\right)
\end{equation}
and so ${\bf u}$ is completely determined.
\noindent \item The proof is now concluded after we have established 
 \eqref{EDOFourierFwd}. Since both sides of it are continuous in $(t,\xi)$,
it will be enough to show the equality as $\shs'(\R^d)$-valued.
This can be done differentiating \eqref{WeakOUPDE}, considered
as an equality in $\shs'(\R^d)$.
%
For this we will apply  Lemma \ref{weakDer} below
setting $\Phi := \shf {\bf u}\left(t\right)$
for every fixed
 $t \in [0,T]$ and differentiating in time.
We set $\Phi_t(\xi) = \shf {\bf v}(t)(\xi), \ \xi \in \R^d$
and $\Phi_t(\varphi) = \int_{\R^d} \varphi(\xi) \Phi_t(\xi) {\mathrm d}\xi,
\varphi \in \shs(\R^d)$.
\eqref{EDOFourierFwd} follows from Lemma \ref{weakDer} below
remarking that $\Phi_t$ is compatible
with the one defined in
 \eqref{EPhi}. 
\end{enumerate}

\end{proof}
\begin{lem}\label{weakDer}
  \noindent Let $\Phi \in \shs^{'}\left(\R^d\right), t \in [0,T]$.
 We denote by $\Phi_t$ the element of $\shs^{'}\left(\R^d\right)$ such that for all $\varphi \in \shs\left(\R^d\right)$
\begin{equation}\label{EPhi}
\Phi_t\left(\varphi\right) := \det\left(\shd^{-1}\left(t\right)\right) \Phi\left(\varphi\left(\shd^{-1}\left(t\right)\cdot\right)\right).
\end{equation}
Then, for all $t \in [0,T]$
\begin{equation}\label{EDerivS}
\Phi_t(\varphi)  = \Phi(\varphi) - \sum^{d}_{i=1}\int^{t}_{0}\left(\partial_i\Phi\right)_s\left(x \mapsto \left(C\left(s\right)^{\top}\shd\left(s\right)x\right)_i\varphi(x)\right)ds.
\end{equation}
\end{lem}
\begin{proof}

  \noindent 
  We begin with the case $\Phi \in \shs\left(\R^d\right)$ 
(or only $\shc^\infty\left(\R^d\right)$). In this case,
\be
\Phi_t\left(x\right) = \Phi\left(\shd\left(t\right)x\right),
\ x \in \R^d, t \in [0,T].
\ee
Hence, for every $t \in [0,T]$
\begin{align*}
\frac{d}{dt}\Phi_t\left(x\right) &{}= \left<\frac{d}{dt}\left(\shd\left(t\right)x\right),\nabla\Phi\left(\shd\left(t\right)x\right)\right> \\
&{}= - \left<C\left(t\right)^{\top}\shd\left(t\right)x,\nabla\Phi\left(\shd\left(t\right)x\right)\right>\\
&{}= - \sum^{d}_{i=1}\left(C\left(t\right)^{\top}\shd\left(t\right)x\right)_i\left(\partial_i\Phi\right)_t\left(x\right),  
\end{align*}
 Now, coming back to the general case, let $\Phi \in \shs'\left(\R^d\right) $ and $\left(\phi_\epsilon\right)_{\epsilon > 0}$ be a sequence of mollifiers in $\shs\left(\R^d\right)$, converging to the Dirac measure.
 Then for all $\epsilon > 0$, the function $\Phi*\phi_\epsilon : x \mapsto \Phi\left(\phi_\epsilon\left(x-\cdot\right)\right)$ belongs to $\shs'\left(\R^d\right)\cap\shc^\infty\left(\R^d\right)$. By the first part of the proof,
 \eqref{EDerivS} holds replacing  $\Phi$ with $ \Phi * \varphi_\varepsilon$.
Now, this converges to $\Phi$ in $\shs'\left(\R^d\right)$ when $\epsilon$ tends to $0+$. 
\eqref{EDerivS}  follows sending $\epsilon$ to $0+$.
Indeed, for all $\varphi \in \shs\left(\R^d\right)$, $t \in [0,T]$, setting $\check{\phi_\epsilon} : y \mapsto \phi_\epsilon(-y)$, we have
\begin{align*}
\Phi_t\left(\varphi\right) &{}= \lim\limits_{\epsilon \to 0+}\int_{\R^d}\varphi(x)\left(\Phi*\phi_\epsilon\right)_t\left(x\right)dx\\ 
&{}= \lim\limits_{\epsilon \to 0+}\int_{\R^d}\varphi(x)\Phi*\phi_\epsilon\left(x\right)dx -  \lim\limits_{\epsilon \to 0+}\sum^{d}_{i=1}\int^{t}_{0}\det\left(\shd^{-1}\left(s\right)\right)\int_{\R^d}\left(C\left(s\right)^{\top}x\right)_i\varphi\left(\shd^{-1}\left(s\right)x\right)\partial_i\Phi*\phi_\epsilon(x)dxds\\ 
&{}= \Phi(\varphi) - \lim\limits_{\epsilon \to 0+}\sum^{d}_{i=1}\int^{t}_{0}\det\left(\shd^{-1}\left(s\right)\right)\partial_i\Phi\left(\left(\left(C\left(s\right)^{\top}\cdot\right)_i\varphi\left(\shd^{-1}\left(s\right)\cdot\right)\right)*\check{\phi_\epsilon}\right)ds\\
&{}= \Phi(\varphi) - \sum^{d}_{i=1}\int^{t}_{0}\det\left(\shd^{-1}\left(s\right)\right)\partial_i\Phi\left(\left(C\left(s\right)^{\top}\cdot\right)_i\varphi\left(\shd^{-1}\left(s\right)\cdot\right)\right)ds\\
&{}= \Phi(\varphi) - \sum^{d}_{i=1}\int^{t}_{0}\left(\partial_i\Phi\right)_s\left(x \mapsto \left(C\left(s\right)^{\top}\shd\left(s\right)x\right)_i\varphi\left(x\right)\right)ds.\\
\end{align*}
 To conclude, it remains to justify the commutation between the limit in $\epsilon$ and the integral in time from line 3 to line 4 using Lebesgue's dominated convergence theorem. On the one hand, for a given $i \in [\![1,d]\!]$, the fact $\partial_i \Phi$ belongs to $\shs'\left(\R^d\right)$ implies that there exists $C > 0$, $N \in \N$ such that for all $\varphi \in \shs\left(\R^d\right)$
\be
\left|\partial_i \Phi\left(\varphi\right)\right| \leq C \sup_{\left|\alpha\right| \leq N}\sup_{x\in\R^d}\left(1 + \left|x\right|^2\right)^N\left|\partial^{\alpha}_x\varphi(x)\right|,
\ee 
 see Chapter 1, Exercise 8 in \cite{rudin}. On the other hand, the quantities
\be 
\sup_{x\in\R^d}\left(1 + \left|x\right|^2\right)^N\left|\partial^{\alpha}_x\left(x_j\varphi(\shd^{-1}\left(s\right)\cdot)\right)*\check{\phi_\epsilon}\right|
\ee
 are bounded uniformly in the couple $\left(s,\epsilon\right)$, for all $j \in [\![1,d]\!]$, $\alpha \in \mathbb{N}^d,$
taking also into account that the function $s \mapsto \shd^{-1}(s)$
is continuous and therefore bounded.
Since $C$ is also  continuous on $[0,T]$, we are justified to
use Lebesgue's dominated convergence theorem.


\end{proof}
We state now the main result of this section.

\begin{thm}\label{BwdOU_Uniq} 
({\bf Uniqueness: the case of OU semigroup}).

 For all $\mu \in \shm_f\left(\R^d\right)$, the PDE \eqref{EDPTerm0} with terminal value $\mu$ admits at most one $\shm_f\left(\R^d\right)$-valued solution in the sense of Definition \ref{Def}.
\end{thm}
\begin{proof}
\noindent Let $\mu \in \shm_f\left(\R^d\right)$ and ${\bf u}$ be a solution of \eqref{EDPTerm0} with terminal value $\mu$. Then, ${\bf u}$ solves the PDE
 \eqref{Fokker} with initial value ${\bf u}\left(0\right)$. As a consequence, 
by
  \eqref{FourierExplicitOU} appearing at the end of the proof of Proposition \ref{FwdOU_Uniq}, for all $\xi \in \R^d$,
\be 
\shf \mu \left(\xi\right) = e^{-\int^{T}_{0}\frac{\left|\sigma\left(s\right)^{\top}\xi\right|^2}{2} ds}\shf {\bf u}\left(0\right)\left(\shd^{-1}\left(T\right)\xi\right),
\ee
so that 
\be 
\shf {\bf u}\left(0\right)(\xi) = 
 e^{\int^{T}_{0}\frac{\left|\sigma\left(s\right)^{\top}\xi\right|^2}{2}ds}\shf \mu \left(\shd\left(T\right)\xi\right).
\ee
 Hence, ${\bf u}\left(0\right)$ is entirely determined by $\mu$ and Proposition \ref{FwdOU_Uniq} gives the result.
\end{proof}

\section{McKean SDE related to the PDE with terminal condition}
\label{S4}

\setcounter{equation}{0}

\noindent In this section, we concentrate on the   
 analysis of the well-posedness of the
McKean SDE \eqref{MKIntro} that we relate to the 
PDE \eqref{EDPTerm0}.
The existence results for the SDE \eqref{MKIntro}  will be based on two 
pillars: the reachability condition constituted by the existence of a solution
of the Fokker-Planck PDE with terminal condition  and
 the time-reversal techniques of \cite{haussmann_pardoux}. They follow from general 
statements of Section \ref{MKEX}, in the Appendix.
 The uniqueness results for 
 the SDE \eqref{MKIntro}
will be a consequence of results stated in 
Section \ref{MKUNIQ}.

\subsection{Preliminary considerations} \label{Prelim}

\smallbreak
\noindent Regarding $b: [0,T]\times\mathbb{R}^d \mapsto \mathbb{R}^d$, 
$\sigma: [0,T]\times\mathbb{R}^d \mapsto M_d\left(\mathbb{R}\right)$, we
set 
$\widehat{b} := b\left(T-.,\cdot\right)$, $\widehat{\sigma} := \sigma\left(T-.,\cdot\right)$,
$\widehat {\Sigma}:= \widehat{\sigma}^\top \widehat{\sigma}.$
\smallbreak

\noindent Given a probability-valued function ${\bf p}: [0,T] \rightarrow 
\shp(\R^d)$,
 we denote by $p_t$ the density of ${\bf p}\left(t\right)$, for $t \in [0,T]$,
 whenever it exists.
 In this section $\mu$ will denote the terminal condition
 of the PDE \eqref{EDPTerm0} supposed to be a probability.
 For the McKean type SDE \eqref{MKIntro}, remarking that $\mu = \bar \mu$,
 we consider the following notion of solution.
\begin{defi} 
\label{MKSol}
  On a given filtered probability space $\left(\Omega, \shf, \left(\mathcal{F}_t\right)_{t\in[0,T]}, \P\right)$
  equipped with an $d$-dimensional $\left(\mathcal{F}_t\right)_{t\in[0,T]}$-Brownian motion $\beta$,
  a  {\bf solution} of the SDE \eqref{MKIntro} is a
  couple $\left(Y,{\bf p}\right)$
  fulfilling \eqref{MKIntro}, 
  such that
 $Y$ is $\left(\shf_t\right)_{t\in[0,T]}$-adapted 
 and such that for all $i \in [\![1,d]\!]$, all compact $K \subset \R^d$, all $\tau < T$
 \begin{equation} \label{IdInt}
 \int^{\tau}_{0}\int_{K}\left| \mathop{div_y}\left(\widehat{\Sigma}_{i.}\left(r,y\right)p_{r}\left(y\right)\right) \right|   dydr <  \infty.
 \end{equation}
\end{defi} 
\begin{rem} \label{RDefMK}
  For a given solution $\left(Y,{\bf p}\right)$ of equation \eqref{MKIntro},
  identity \eqref{IdInt} appearing in Definition \ref{MKSol} implies
  in particular
  that, for all $i \in [\![1,d]\!]$, all $\tau < T$
\be 
\E \left( \int^{\tau}_{0} \left|\frac{\mathop{div_y}\left(\widehat{\Sigma}_{i.}\left(r,Y_r\right)p_{r}\left(Y_r\right)\right)}{p_{r}\left(Y_r\right)}\right|dr \right)
< \infty.
\ee 
\end{rem}
\noindent The terminology stating that the SDE \eqref{MKIntro} constitutes a probabilistic representation
of the PDE \eqref{EDPTerm0}
 is justified by the result below.
\begin{prop} \label{PProbRep} Suppose $b,\sigma$ locally bounded. If $\left(Y,{\bf p}\right)$
is a solution of \eqref{MKIntro} in the sense of Definition \ref{MKSol}, then ${\bf p}\left(T-\cdot\right)$ is a solution
of \eqref{EDPTerm0}, with $\mu = {\bf p}(0)$ in the sense of Definition \ref{Def}.
\end{prop}

\begin{proof} 
\noindent Let $\left(Y,{\bf p}\right)$
be a solution of \eqref{MKIntro} in the sense of Definition \ref{MKSol} with a Brownian motion symbolized by $\beta$. Let $\phi \in \mathcal{C}^{\infty}_c\left(\mathbb{R}^d\right)$
 and $t\in]0,T]$. It\^o's formula gives
\begin{equation} \label{Ito}
\phi\left(Y_{T-t}\right) = \phi\left(Y_0\right) + \int^{T-t}_{0}\left<\tilde{b}(s,Y_s; p_s),\nabla\phi\left(Y_s\right)\right> + \frac{1}{2}Tr\left(\widehat{\Sigma}\left(s,Y_s\right)\nabla^2\phi\left(Y_s\right)\right)ds + \int^{T-t}_{0}\nabla\phi\left(Y_s\right)^\top\sigma\left(s,Y_s\right)d\beta_s,
\end{equation}
 with 
\be
\tilde{b}\left(s,y; p_s\right) := \left\{\frac{\mathop{div_y}\left(\widehat{\Sigma}_{j.}\left(s,y\right)p_{s}\left(y\right)\right)}{p_{s}\left(y\right)}\right\}_{j\in[\![1,d]\!]} - \widehat{b}\left(s,y\right), \quad \left(s,y\right) \in ]0,T[\times\R^d.
\ee
We now want to take the expectation in identity \eqref{Ito}.
On the one hand,
Remark \ref{RDefMK}, implies that, for all $i \in [\![1,d]\!]$ and
$s\in ]0,T[$ a.e.
\be
\mathbb{E} \left (\left \vert \frac{\mathop{div_y}\left(\widehat{\Sigma}_{i.}\left(s,Y_s\right)p_{s}\left(Y_s\right)\right)}{p_{s}\left(Y_s\right)}\partial_{i}\phi\left(Y_s\right)\right\vert \right) < \infty.
\ee
 On the other hand 
\be 
\label{E42bis}
\int^{T}_{0}\mathbb{E}\left\{Tr\left(\widehat{\Sigma}\left(s,Y_s\right)\nabla^2\phi\left(Y_s\right)\right)\right\}ds = \sum^{d}_{i,j=1}\int^{T}_{0}\int_{\R^d}\widehat{\Sigma}_{ij}\left(s,y\right)\partial_{ij}\phi\left(y\right)p_s\left(y\right)dyds.
\ee
 Previous expression is finite since $\Sigma$ is bounded on compact sets
and the partial derivatives of $\phi$ have compact supports.
With similar arguments we prove that
$ \int_0^T ds    \mathbb{E}\left\vert\left<\widehat{b}\left(s,Y_s\right),\nabla\phi\left(Y_s\right)\right>\right
\vert < \infty,\ s \in ]0,T[.$
 Moreover, fixing $s\in ]0,T[$ a.e., integrating by parts we have
\begin{align} \label{E42quater}
\mathbb{E}\left\{\left<\tilde{b}\left(s,Y_s;  p_s \right),\nabla\phi\left(Y_s\right)\right>\right\} &{}= \sum^{d}_{k,j=1}\int_{\R^d}\partial_k\left(\widehat{\Sigma}_{jk}\left(s,y\right)p_{s}\left(y\right)\right)\partial_j\phi\left(y\right)dy- \int_{\R^d}\left<\widehat{b}\left(s,y\right),\nabla\phi\left(y\right)\right>p_{s}\left(y\right)dy\\  
&{}= -\int_{\R^d}Tr\left(\widehat{\Sigma}\left(s,y\right)\nabla
^2\phi\left(y\right)\right)p_{s}\left(y\right)dy- \int_{\R^d}\left<\widehat{b}\left(s,y\right),\nabla\phi\left(y\right)\right>p_{s}\left(y\right)dy. \nonumber 
\end{align}
 Now, the quadratic variation of the local martingale $M^{Y} := \int^{\cdot}_{0}\nabla\phi\left(Y_s\right)^{\top}\sigma\left(s,Y_s\right)d\beta_s$ yields
\be
\left[M^{Y}\right] = \int^{\cdot}_{0}\nabla\phi\left(Y_s\right)^{\top}\Sigma\left(s,Y_s\right)\nabla\phi\left(Y_s\right)ds.
\ee
 We remark in particular that $\mathbb{E}\left(\left[M^Y\right]_T\right) < \infty$ since $\Sigma$ is bounded on compact sets and $\phi$ has compact support. 
This shows $M^Y$ is a true (even square integrable) martingale and all terms involved in \eqref{Ito} are integrable.

At this point we  evaluate    the expectation in \eqref{Ito} taking
the considerations above together with
\eqref{E42bis} and \eqref{E42quater} into account.
We obtain 
\be
\mathbb{E}\left(\phi\left(Y_{T-t}\right)\right) = \int_{\R^d}\phi\left(y\right)\mu\left(dy\right) - \int^{T-t}_{0}\int_{\R^d}L_{T-s}\phi\left(y\right)p_{s}\left(y\right)dyds. 
\ee
Applying  the change of variable $t \mapsto T-t$, we finally obtain the
identity
\be
\int_{\R^d}\phi\left(y\right)p_{T-t}\left(y\right)dy = \int_{\R^d}\phi\left(y\right)\mu \left(dy\right) - \int^{T}_{t}\int_{\R^d}L_{s}\phi\left(y\right)p_{T-s}\left(y\right)dyds, 
\ee
 which means that ${\bf p}\left(T-\cdot\right)$ solves the PDE \eqref{EDPTerm0} in the sense of Definition \ref{Def} with terminal value $\mu$.
\end{proof}

\subsection{Notion of existence and uniqueness 
for the McKean SDE in a given class }

\noindent We  provide the different notions of existence and uniqueness for \eqref{MKIntro} we will use in the sequel.
\begin{defi} \label{MKDSol}
 Let $\sha$ be a class
  of measure-valued functions from $[0,T]$ to $\shp\left(\R^d\right)$. 
\begin{enumerate}
\item We say that the SDE \eqref{MKIntro} admits {\bf existence in law} in 
 $\sha$, 
 if there exists a complete filtered probability space $\left(\Omega, \shf, \left(\mathcal{F}_t\right)_{t\in[0,T]}, \P\right)$ equipped with an $m$-dimensional $\left(\mathcal{F}_t\right)_{t\in[0,T]}$-Brownian motion $\beta$ and a couple $\left(Y,{\bf p}\right)$ solution of \eqref{MKIntro} in the sense of Definition \ref{MKSol} such that ${\bf p}$ belongs to $\sha$.
\item
  Let $\left(Y^1,{\bf p^1}\right)$, $\left(Y^2,{\bf p^2}\right)$  be two solutions
of \eqref{MKIntro}  in the sense of Definition \ref{MKSol} associated to some complete filtered probability spaces $\left(\Omega^1, \shf^1, \left(\mathcal{F}^1_t\right)_{t\in[0,T]}, \P^1\right)$, $\left(\Omega^2, \shf^2, \left(\mathcal{F}^2_t\right)_{t\in[0,T]}, \P^2\right)$ respectively, equipped with Brownian motions $\beta^1,\beta^2$ respectively and such that ${\bf p^1},{\bf p^2}$ belong to $\sha$.
 We say that \eqref{MKIntro} admits {\bf uniqueness in law}
   in $\sha$, if  $Y^1_0,Y^2_0$ have the same law implies that $Y^1,Y^2$ have
 the same law.
\item We say that \eqref{MKIntro} admits {\bf strong existence} in 
  $\sha$
if for  any complete filtered probability space $(\Omega, \shf, \left(\mathcal{F}_t\right)_{t\in[0,T]}, \P)$ equipped with an $m$-dimensional $\left(\mathcal{F}_t\right)_{t\in[0,T]}$-Brownian motion $\beta$, there exists a solution $\left(Y,{\bf p}\right)$ of equation \eqref{MKIntro} in the sense of Definition \ref{MKSol} such that ${\bf p}$ belongs to $\sha$.
\item We say that \eqref{MKIntro} admits {\bf pathwise uniqueness}
  in  $\sha$ of
  if for any complete filtered probability space $(\Omega, \shf, \left(\mathcal{F}_t\right)_{t\in[0,T]}, \P)$ equipped with an $m$-dimensional $\left(\mathcal{F}_t\right)_{t\in[0,T]}$-Brownian motion $\beta$, for any solutions
  $\left(Y^1,{\bf p^1}\right)$, $\left(Y^2,{\bf p^2}\right)$ of \eqref{MKIntro}  in the sense of Definition \ref{MKSol}
such that $Y^1_0 = Y^2_0, \ \mathbb{P}\rm{-a.s.}$ and ${\bf p^1}, {\bf p^2}$ belong to $\sha$, we have $Y^1 = Y^2, \ \mathbb{P}\rm{-a.s.}.$
\item If the mention to a specific class $\sha$ is omitted 
as far as uniqueness (in law or pathwise), the class $\sha$ is
the one of all possible probability valued functions
 verifying \eqref{IdInt}.
\end{enumerate}

\end{defi}
 We finally define the sets in which we will
 formulate existence and uniqueness results in the sequel.
\smallbreak
\begin{notation} \label{NAC1_2}
  \begin{enumerate}
  \item For a given $\shc \subseteq \shp\left(\R^d\right)$, $\sha_{\shc}$
    denotes
    the set of measure-valued functions ${\bf p}$ from $[0,T]$ to $\shp\left(\R^d\right)$  such that ${\bf p}\left(T\right)$ belongs to $\mathcal{C}$.
 Furthermore, for a given measure-valued function 
${\bf p}: [0,T] \mapsto \shp\left(\R^d\right)$, we will write
\begin{equation} \label{EBP}
b(t,\cdot; {\bf p}_t ) := \left\{\frac{\mathop{div_y}\left(\widehat{\Sigma}_{i.}p_t\right)}{p_t}\right\}_{i\in[\![1,d]\!]},
\end{equation}
 for almost all $t \in [0,T]$ whenever $p_t$ exists and the 
right-hand side quantity of \eqref{EBP} is well-defined.
The function $(t,x) \mapsto b(t,x; {\bf p}_t)$ is defined on $[0,T] \times \R^d$ with values in  $\R^d$. 
\item
Let $\sha_1$ (resp. $\sha_2$)  denote the set of measure-valued functions from $[0,T]$ to $\shp\left(\R^d\right)$ ${\bf p}$ such that, for all $t \in [0,T[$, ${\bf p}\left(t\right)$ admits a density $p_t$ with respect to the Lebesgue measure on $\R^d$ and such that
$(t,x) \mapsto b(t,x; {\bf p}_t )$
is locally bounded (resp. is locally Lipschitz in space with linear growth) on $[0,T[\times\R^d$.
\end{enumerate}
\end{notation}

\subsection{Well-posedness for the McKean SDE: the 
 bounded  coefficients case}
\label{SExamples44}

\smallbreak

 In this section, we state a significant result about 
existence  
and uniqueness in law together with pathwise uniqueness
 for the SDE \eqref{MKIntro}.
We exploit here in particular the uniqueness results related to
the PDE \eqref{EDPTerm0} obtained in Section \ref{S32} and Section \ref{SGP}.
As  far as uniqueness
is concerned, given a solution 
$(Y, {\bf p})$ of \eqref{MKIntro}, we insist
that the basic idea
consists in showing that ${\bf p}$ solves \eqref{EDPTerm0},
  see Proposition \ref{PProbRep}. At this point $Y$ solves
an ordinary SDE and we only need to show that the coefficients
fulfill the assumptions which guarantee uniqueness, see e.g.
Lemma \ref{FriedAr}.
On the other hand, the existence results 
for \eqref{MKIntro}
are based on the techniques of \cite{haussmann_pardoux}
of determining the dynamics of the time-reversal of a diffusion.

\smallbreak
\noindent  We formulate the following
hypothesis for the couple $(b, \Sigma)$, where we recall that $\Sigma =
\sigma \sigma^\top$.
\begin{ass}\label{smoothness}
  \smallbreak
  $\Sigma: [0,T] \times \R^d \rightarrow M_d(\R),$
   $b: [0,T] \times \R^d \rightarrow \R^d$)
  are Borel functions such that the following holds.
\begin{itemize}
\item  For each $t \in |0,T],$
  $\left(\nabla_xb_i(t,\cdot) \right)_{i \in [\![1,d]\!]}$, 
 $\left(\nabla_x\Sigma_{ij}(t,\cdot)\right)_{i,j \in [\![1,d]\!]}$ 
 exist and they are continuous;
\item For each $t \in |0,T],$ $\nabla^2_x\Sigma(t,\cdot)$ exists
  and $\nabla^2_x\Sigma$ 
  is H\"{o}lder continuous in space with some exponent
     $\alpha \in ]0,1[$ uniformly in time.
\item $\nabla \Sigma$ and $\nabla b$ are uniformly bounded.
   \end{itemize}
   \end{ass}
\begin{ass}\label{smoothness1}
$\Sigma$ is supposed to be H\"{o}lder continuous in time.

\end{ass}
The first step consists in proving existence and uniqueness in law for
 the SDE \eqref{MKIntro} in the 
class $\sha_1.$
For this we will state
 a fundamental lemma whose proof will appear
in the Appendix.
\begin{lem} \label{FriedAr}
\noindent Suppose the validity of
Assumptions \ref{Zvon1a},
\ref{Zvon3},
 \ref{smoothness}
and \ref{smoothness1}. Then, for all $\nu \in \shp\left(\R^d\right)$, ${\bf u}^\nu\left(t\right)$ admits a density $u^{\nu}\left(t,\cdot\right) \in \shc^1\left(\R^d\right)$ for all $t \in ]0,T]$.
Furthermore, for each compact $K$ of $]0,T] \times \R^d$, 
there are strictly positive constants $C^K_1, C^K_2, C^K_3$,
also depending on $\nu$ such that 
\begin{eqnarray} 
C^K_1 \le      u^\nu\left(t,x\right) &\leq&  C^K_2  \label{dens}   \\
  \left|\partial_iu^\nu\left(t,x\right)\right| &\leq& C^K_3, \ 
  i \in  [\![1,d]\!], \label{DerDens}
\end{eqnarray} 
for all $(t,x) \in K$.
\end{lem}
\begin{rem} \label{Runu}
 Under Assumptions
\ref{Zvon1a}, \ref{Zvon1b}
(which is a consequence of Assumptions \ref{smoothness}
and \ref{smoothness1})
 together with
  \ref{Zvon3}, 
  for every $\nu \in \shp(\R^d),$
  by Lemma \ref{LC313},
  there exists a unique $ \shp\left(\R^d\right)$-valued solution
  ${\bf u}^\nu$ of the PDE \eqref{Fokker}.\\
\end{rem}

\begin{lem} \label{P49}
  Let 
   $\mu $ be the probability measure introduced at the beginning of
  Section \ref{Prelim}.
  Suppose that  $\mu =
  {\bf u}^\nu\left(T\right)$ for some $\nu \in \shp\left(\R^d\right)$.
We assume the following.
\begin{enumerate}
\item  Assumptions \ref{Lip1d}, \ref{Zvon1a},
 \ref{Zvon3} and
 \ref{smoothness}.
\item ${\bf u}^\nu\left(t\right)$ admits a density
 $u^{\nu}\left(t,\cdot\right) \in  W^{1,1}_{\rm loc}(\R^d),$
for all $t \in ]0,T]$.
\item  For each compact $K$ of $]0,T] \times \R^d$, 
there are strictly positive constants $C^K_1, C^K_2, C^K_3$,
also depending on $\nu$ such that \eqref{dens} and \eqref{DerDens} hold
$\forall (t,x) \in K$.
\end{enumerate}
Then the SDE \eqref{MKIntro}  admits existence in law in $\sha_1$.
\end{lem}
A consequence of the two lemmata above is the proposition below,
which states in particular existence in law in $\sha_1$.
\begin{prop}\label{CP49}
We suppose the validity of Assumptions \ref{Lip1d}, \ref{Zvon1a},
 \ref{Zvon3},
 \ref{smoothness}
  and 
  \ref{smoothness1}.
\begin{enumerate}
\item Suppose the existence of $\nu \in \shp(\R^d)$
  such that ${\bf u}^{\nu}(T) = \mu$.
   Then, 
the SDE \eqref{MKIntro}  admits existence in law in $\sha_1$.
Moreover, if $\nu$ is a Dirac mass, existence in law occurs in
$\sha_{\left(\delta_x\right)_{x\in \R^d}}\cap\sha_1$.
\item Otherwise  \eqref{MKIntro} does not admit existence in law.
 \end{enumerate} 
\end{prop}
\begin{rem} \label{RReach}
For a class of coefficients $b, \Sigma$, an interesting problem would be
to determine  the {\it reachability set} of possible $\mu$, i.e. of the set of
$\mu$ for which there exists $\nu$ such that
$ \mu  = {\bf u}^\nu$. This however goes beyond the scope of our paper. 
\end{rem}
 \begin{prooff}\ (of Proposition \ref{CP49}).
   \begin{enumerate}
     \item
   The first part is a direct consequence of Lemma \ref{FriedAr}, 
Lemma \ref{P49} and expression \eqref{EBP}.
 If in addition, $\nu$ is a Dirac mass,
   then ${\bf u}^\nu\left(0\right)$ belongs to $\shc := \left(\delta_x\right)_{x\in\R^d}$, hence existence in law occurs in $\sha_\shc\cap\sha_1$ by Proposition \ref{MKEx_Prop}
in the Appendix.
 \item Otherwise suppose ab absurdo that $\left(Y,{\bf p}\right)$ is a solution
   of the SDE\eqref{MKIntro}.
 By Proposition \ref{PProbRep}  ${\bf p}\left(T-\cdot\right)$ is a
  solution of the PDE \eqref{EDPTerm0}. We set $ \nu_0 = {\bf p}(T)$
  so that ${\bf p}(T-\cdot)$  verifies also the PDE \eqref{Fokker} with initial value
  $ \nu_0$. Since, by Lemma \ref{LC313} uniqueness holds for
  \eqref{Fokker}, it follows that  ${\bf p}(T-\cdot) =  {\bf u}^{\nu_0}$
  which concludes the proof of item 2.
     \end{enumerate}
 \end{prooff}
 \begin{prooff} \ (of Lemma \ref{P49}).
  \noindent Suppose $\mu = {\bf u}^\nu\left(T\right)$ for some $\nu \in \shp\left(\R^d\right)$.
 We recall that Property \ref{GH1} holds with respect to $\shc := \shp\left(\R^d\right)$ by
Lemma \ref{LC313}.
In view of applying again Proposition \ref{MKEx_Prop} stated in the Appendix,
we  need to check the validity of Property \ref{MKEx_1}
with respect  to $\shc$ and Property \ref{MKEx_2}.
 Property \ref{MKEx_1} is verified by ${\bf u} =   {\bf u}^\nu$. 
Indeed the function ${\bf u}^\nu$ is a $\shp\left(\R^d\right)$-valued solution
 of the PDE \eqref{EDPTerm0}
 with terminal value $\mu$ and such that ${\bf u}^\nu\left(0\right)$
 belongs to $\shc$.
 Condition \eqref{HP} appearing in Property
 \ref{MKEx_1} is satisfied with ${\bf u} = {\bf u}^\nu$
 thanks to the right-hand side of inequalities \eqref{dens} and \eqref{DerDens}
 and the fact that $\Sigma$ is bounded.
 Hence Property \ref{MKEx_1} holds with respect to
 $\shc$. 
 It remains to show Property \ref{MKEx_2} holds
i.e. that 
$$(t,x) \mapsto 
 \frac{\mathop{div_x} \left(\widehat{\Sigma}_{i.}(t,x)
 u^\nu(T-t,x)\right)}{u^\nu(T-t,x)}$$
is locally bounded on $[0,T[ \times \R^d$.
To achieve this, we fix $i \in [\![1,d]\!]$ and a bounded open 
subset $\mathcal{O}$ of
$[0,T[ \times \R^d$.
For $(t,x) \in \mathcal{O}$ we have
 $$ \left|\frac{\mathop{div_x}\left(\widehat{\Sigma}_{i.}\left(t,x\right) u^\nu\left(T-t,x\right)\right)}{u^\nu\left(T-t,x\right)}\right|  {} \leq \left|\mathop{div_x}\left(\widehat{\Sigma}_{i.}\left(t,x\right)\right)\right| +
 \left|{\widehat \Sigma}_{i.}\left(t,x\right)\right|\frac{\left|\nabla_x u^\nu\left(T-t,x\right)\right|}{u^\nu\left(T-t,x\right)}. 
$$
 The latter quantity  is locally bounded in $t,x$ thanks to
 the boundedness of $\Sigma,\mathop{div_x}\left(\widehat{\Sigma}_{i.}\right)$
and inequalities \eqref{dens} and \eqref{DerDens}.
 Hence, Property \ref{MKEx_2} holds. 
The application of  Proposition \ref{MKEx_Prop} ends the proof.
\end{prooff}

 \begin{prop} \label{TExUniq}
 ({\bf The McKean SDE: well-posedness in the case of Dirac initial conditions.})

   Suppose the validity of  Assumptions \ref{Lip1d}, \ref{Zvon1a},
 \ref{Zvon3},
 \ref{smoothness} and \ref{smoothness1}.
 The following results hold. 
\begin{enumerate}
\item Let us suppose $d = 1$. Suppose $\mu$ equals
${\bf u}^{\delta_{x_0}} \left(T\right)$ for some
  $x_0 \in \R$. 
Then \eqref{MKIntro} 
admits existence and uniqueness in law in
 $\sha_{\left(\delta_x\right)_{x\in \R^d}}\cap\sha_1$, pathwise uniqueness in $\sha_{\left(\delta_x\right)_{x\in \R^d}}\cap\sha_2$.
\item Let $d \geq 2$. There is a maturity $T$ sufficiently small
(only depending on the Lipschitz constant of $b, \sigma$)
such that the following result holds.
 Suppose $\mu$ equals
${\bf u}^{\delta_{x_0}} \left(T\right)$ for some
  $x_0 \in \R^d$. 
 Then \eqref{MKIntro} 
admits existence and uniqueness in law in
 $\sha_{\left(\delta_x\right)_{x\in \R^d}}\cap\sha_1$, pathwise uniqueness in $\sha_{\left(\delta_x\right)_{x\in \R^d}}\cap\sha_2$.

\end{enumerate}
\end{prop}
\begin{proof} 
\noindent
By Assumptions \ref{Lip1d}, \ref{Zvon1a},
\ref{Zvon3},
 \ref{smoothness} and \ref{smoothness1}, Proposition \ref{CP49}
implies that the SDE \eqref{MKIntro}
admits existence in law in the two cases in the specific classes.
 To check the uniqueness in law and pathwise uniqueness results, we wish to apply Corollary \ref{Coro} stated in the Appendix. It suffices now to check 
Property \ref{APDETerm} 
 with respect to $ (\delta_x)_{x\in \R}$,  for the separate
two cases.
\begin{enumerate}
\item Fix $x_0 \in \R^d$.
This will follow from Proposition \ref{propLip1}
that holds under Assumption \ref{Lip1d}.
\item We proceed as  for previous case but applying Theorem \ref{propLipd}
instead of Proposition \ref{propLip1}.
\end{enumerate}
\end{proof}
\noindent 
Previous Proposition \ref{TExUniq} provides uniqueness
in law only among the solutions $(Y, {\bf p})$ such
${\bf p}$ belongs to a subclass of $\sha_1$.
We state now the two most important results of the section
which in particular provide uniqueness in law for the SDE \eqref{MKIntro}
among all possible solutions.


\begin{thm} \label{TExUniqBis}
  ({\bf The McKean SDE: well-posedness  in the case of non-degenerate time-homogeneous
    coefficients.})
  
  Suppose $b,\sigma$ are time-homogeneous,
 Assumptions \ref{Lip1d}, \ref{Zvon1a},
 \ref{Zvon3},
 \ref{smoothness}
  and suppose
  there is  $\nu \in \shp\left(\R^d\right)$ (a priori not known)
 such that
$\mu = {\bf u}^\nu\left(T\right)$.
\begin{enumerate}
\item  The SDE \eqref{MKIntro} admits existence and uniqueness in law.
  Moreover existence in law holds in $\sha_1$.
\item  \eqref{MKIntro} admits pathwise uniqueness in $\sha_2$.
\end{enumerate}
\end{thm}
\begin{proof}
\begin{enumerate}
\item
  \begin{enumerate}
    \item First, Assumption \ref{smoothness1} trivially holds since
  $b,\sigma$ are time-homogeneous. Then, point 1. of Proposition \ref{CP49} 
implies that
the SDE \eqref{MKIntro} admits existence in law (in $\sha_1$).
  \smallbreak
\item
  Let $\left(Y,{\bf p}\right)$ be a solution of \eqref{MKIntro}. Proceeding as in the proof of item 2. of Proposition \ref{CP49}, we obtain that ${\bf p}(T-\cdot) =  {\bf u}^{ \nu_0}$ with $\nu_0 = {\bf p}\left(T\right)$.
  Then, Lemma \ref{FriedAr}
  shows  that ${\bf p}$ belongs to $\sha_1$, see
\eqref{EBP} in Notation \ref{NAC1_2}.
\item To conclude it remains to show
  uniqueness in law in $\sha_1$.
    For this we wish to apply point 1. of Corollary \ref{Coro} in the Appendix.
 To achieve this, we check Property  
 \ref{APDETerm} with respect to 
 $\shp(\R^d)$.
This is a consequence of  
 Assumptions \ref{Zvon1a}, \ref{Zvon1b},
 \ref{Zvon3} and \ref{Lun1}
and  Theorem \ref{P315}.
 This concludes the proof of item 1.
\end{enumerate}

\item Concerning pathwise uniqueness in $\sha_2$, we proceed as for uniqueness in law but applying point 2. of Corollary \ref{Coro} in the Appendix. This is valid since
  $\sigma$ are bounded and Lipschitz by Assumptions \ref{Lip1d}, \ref{Zvon1a}
  and \ref{Zvon1b}.
\end{enumerate}
\end{proof}
In the result below we extend Theorem \ref{TExUniqBis}
to the case when the coefficients $b,\sigma$ are piecewise
time-homogeneous.

\begin{thm} \label{TC313}
  ({\bf The McKean SDE: well-posedness with non-degenerate piecewise
    time-homogeneous
    coefficients.})

  Let $n \in \N^*$.
Let 
$ 0 = t_0 < \ldots < t_n = T$ be a partition. 
For  $k \in [\![2,n]\!]$ (resp. $k=1$) we
denote $I_k = ]t_{k-1},t_k]$ (resp. $[t_{0},t_1]$).
Suppose that
the following holds.
\begin{enumerate}
\item
For all $k \in [\![1,n]\!]$ 
   the restriction of $\sigma$ (resp. $b$)  to $I_k \times \R^d$
is a time-homogeneous function $\sigma^k: \R^d \rightarrow M_{d}(\R)$
(resp. $b^k: \R^d \rightarrow \R^d$).
\item Assumptions \ref{Zvon1a} and \ref{Zvon3}.
\item $\sigma$ is Lipschitz (in space uniformly in time). 
\item Assumption \ref{smoothness} holds for the couples
  $(b^k, \Sigma^k)$.
\end{enumerate}
Suppose 
$\mu$ equals ${\bf u}^\nu(T)$ for some $\nu \in \shp\left(\R^d\right)$.
Then SDE
\eqref{MKIntro} admits
existence and uniqueness in law. Moreover, existence in law holds
in $\sha_1$.
\end{thm}
\begin{rem} \label{RC313}
  A similar remark as in Proposition \ref{CP49} holds
  for the Theorems \ref{TExUniqBis} and \ref{TC313}.
  If there is no $\nu \in \shp(\R^d)$
  such that ${\bf u}^{\nu}(T) = \mu$,
   then \eqref{MKIntro} does not admit existence in law. 
\end{rem}

\begin{prooff} \
(of Theorem \ref{TC313}). 
  \noindent We recall that by Lemma \ref{LC313}, ${\bf u}^{\nu_0}$
  is well-defined for all $\nu_0 \in \shp\left(\R^d\right)$.
\begin{enumerate}
\item We first show that 
 ${\bf u}^{\nu_0}$ verifies \eqref{dens} and \eqref{DerDens}.
  Indeed, fix $k \in [\![1,n]\!]$.
The restriction ${\bf u_k} $ 
of ${\bf u}^{\nu_0}$ to $\bar I_k$ is a solution ${\bf v}$ of 
the first line of the PDE
\eqref{Fokker}
replacing $[0,T]$ with $\bar I_k$,
$L$ by $L^k$  defined in \eqref{OpLk}, with initial condition $ {\bf v}(t_{k-1}) = {\bf u}^{\nu_0}(t_{k-1})$.
That restriction is even the unique solution, using Lemma \ref{LC313}
replacing $[0,T]$ with $\bar I_k$.
We apply Lemma  \ref{FriedAr} replacing $[0,T]$ with $\bar I_k$,  
taking into account
Assumption \ref{smoothness1},
which holds trivially with respect to $\Sigma^k.$
This implies that ${\bf u}^{\nu_0}$ verifies \eqref{dens} and \eqref{DerDens}
replacing $[0,T]$ with $\bar I_k$,
and  therefore on the whole $[0,T]$.
\item Existence in law in $\sha_1$, follows now  by Lemma \ref{P49}.
\item  It remains to show uniqueness in law.
  Let $\left(Y,{\bf p}\right)$ be a solution of the SDE \eqref{MKIntro}.
  We set  $\nu_0 := {\bf p}\left(T\right)$. Since ${\bf u}^{ \nu_0}$
  and ${\bf p}(T-\cdot)$ solve the PDE \eqref{Fokker}, Lemma \ref{LC313}
  implies that ${\bf p}$ is uniquely determined.
 Similarly as in item 1.(b) of the proof of Theorem 
\ref{TExUniqBis}, 
item 1. of the present proof and Lemma \ref{FriedAr} allow to show
that ${\bf p}$ belongs to $\sha_1$.
\item It remains to show uniqueness in law in $\sha_1$.
For this,  Corollary \ref{C313} implies
Property \ref{APDETerm} in the Appendix with $\shc = \shp(\R^d)$.
Uniqueness of \eqref{MKIntro} in the class $\sha_1$
follows now by Corollary \ref{Coro} in the Appendix, which ends the proof.
\end{enumerate}
\end{prooff}

\subsection{Well-posedness for the McKean SDE: the OU semigroup}

\label{Sex}

 In this section we investigate existence and uniqueness
for the SDE \eqref{MKIntro} in the context of an OU semigroup.
As for Section \ref{SExamples44}, the uniqueness statement for
the related PDE \eqref{EDPTerm0} (see Section \ref{S34}), 
appears to be crucial. 
The only limitation here is that the matrix function
$\Sigma$ has to be invertible, otherwise the additive drift 
in \eqref{MKIntro} would not be defined.

Suppose that $b$ is of the form $ \left(s,x\right) \mapsto C\left(s\right)x$ with $C$ continuous 
from $[0,T]$ to $\R^d$ and $\sigma$ continuous from $[0,T]$ to $M_d\left(\R\right)$. We also suppose that for all $t \in [0,T]$, $\Sigma\left(t\right)$ is invertible.
 We denote by
 $\shc(t):= (\shd(t)^{-1})^\top, t \in [0,T]$
where $\shd$ is the unique solution of \eqref{ED}. 
 Evaluating the
 transposed matrix on both sides of \eqref{ED-1},
we remark that $\shc$ is
   solution of the matrix-valued ODE,  
\be
\shc(t) = I + \int^{t}_{0}C(s)\shc(s)ds, \ t \in [0,T].
\ee
 For a given $x_0 \in \R^d$ and a given t $\in ]0,T]$, we denote by $p^{x_0}_t$ the density of a Gaussian random vector with mean  $m^{x_0}_t = \shc(t)x_0$ and covariance matrix $Q_t = \shc(t)\int^{t}_{0}\shc^{-1}(s)\Sigma(s)\shc^{-1}\left(s\right)^{\top} ds \ \shc(t)^{\top}$. Note that for all $t \in ]0,T]$, $Q_t$ is
strictly positive definite, in particular it is invertible.
 Indeed, for every $t \in [0,T]$,  $\Sigma(t)$ is strictly positive definite. By continuity in $t$,
$\int^{t}_{0}\shc^{-1}(s)\Sigma(s)\shc^{-1}\left(s\right)^{\top}ds$ is also strictly positive definite
and finally the same holds for $Q_t$.
For a given $\nu \in \shp\left(\R^d\right)$, $t \in ]0,T]$,
we set the notation 
\begin{equation}\label{Epnu}
p^\nu_t : x \mapsto \int_{\R^d}p^{x_0}_t\left(x\right)\nu\left(dx_0\right).
\end{equation}
 At this level, we need a lemma.

\begin{lem}\label{OU_lemma}
 Let $\nu \in \shp\left(\R^d\right)$. The measure-valued function $t \mapsto p^\nu_t(x)dx  $
 is the unique solution of the PDE \eqref{Fokker} with initial value $\nu$.
 Consequently it coincides with  ${\bf u}^\nu$. Furthermore, ${\bf u}^\nu\left(T-\cdot\right)$ belongs to $\sha_2$.
\end{lem}

\begin{proof}
\begin{enumerate}
\item 
By Chapter 5, Section 5.6 in \cite{karatshreve}, for every $t \in]0,T]$,
$p^{x_0}_t$ is the density of the random variable $X^{x_0}_t$, where $X^{x_0}$ is the unique strong solution of \eqref{EqLin} with initial value $x_0$.
The mapping $t \mapsto p^{x_0}_t(x) dx$ is a solution of \eqref{Fokker}
by Proposition \ref{PFundam}, with initial condition $\delta_{x_0}$.
Consequently, by superposition, $t \mapsto p^{\nu}_t(x) dx$
is a solution of the PDE
 \eqref{Fokker} with initial value $\nu$.
 
\item By  Proposition \ref{FwdOU_Uniq}, Property \ref{GH1}
  with respect to $\shc =\shp(\R^d)$
  is verified so that
  $t \mapsto p^{\nu}_t(x) dx$ is the unique solution of \eqref{Fokker}
  so that it coincides with
  ${\bf u}^\nu $.
\item It remains to show that ${\bf u}^\nu\left(T-\cdot\right)$ belongs to $\sha_2$, namely that for all $i \in [\![1,d]\!]$
\be
\left(t,x\right) \mapsto \frac{{\mathop {div_x}}\left(\Sigma\left(T-t\right)_{i\cdot}p^{\nu}_{T-t}\left(x\right)\right)}{p^{\nu}_{T-t}\left(x\right)},
\ee
 is locally Lipschitz with linear growth in space on $[0,T[\times\R^d$. 

Fix $i \in [\![1,d]\!]$, $t \in [0,T[$ and $x \in \R^d$.
    Remembering the fact that
 $p^{x_0}_{T-t}$ is a Gaussian law with mean $m^{x_0}_{T-t}$ and covariance matrix $Q_{T-t}$ for a given $x_0 \in \R^d$, we have 
\begin{equation}\label{div_OU}
\frac{{\mathop {div_x}}\left(\Sigma\left(T-t\right)_{i\cdot}p^{\nu}_{T-t}\left(x\right)\right)}{p^{\nu}_{T-t}\left(x\right)} = -\frac{1}{p^\nu_{T-t}\left(x\right)}\int_{\R^d}\left<\Sigma\left(T-t\right)_{i\cdot},Q^{-1}_{T-t}\left(x - m^{x_0}_{T-t}\right)\right>p^{x_0}_{T-t}\left(x\right)\nu\left(dx_0\right).
\end{equation} 
Let $K$ be a compact subset of $]0,T] \times \R^d$;
then there is $M_K > 0$ such that for all $\left(t,x\right) \in  K$, $x_0 \in \R^d$, 
\begin{equation*}
\left|\left<\Sigma\left(T-t\right)_{i\cdot},Q^{-1}_{T-t}\left(x - m^{x_0}_{T-t}\right)\right>p^{x_0}_{T-t}\left(x\right)\right| \leq \left|\Sigma\left(T-t\right)_{i\cdot}\right|\left|\left|Q^{-1}_{T-t}\right|\right|\left|x - m^{x_0}_{T-t}\right|p^{x_0}_{T-t}\left(x\right) \le M_K.
\end{equation*}
This follows because  $t \mapsto \Sigma(T-t)$ and $t \mapsto Q^{-1}_{T-t}$
 are continuous on $[0,T[$ and,
 setting
 $$c_K := \inf\{ t \vert (t,x) \in K\}, \quad
m_K: = \sup_{a \in \R} \vert a\vert \exp\left(- c_K\frac{a^{2}}{2}\right),$$
 we have
$$ |x - m^{x_0}_{T-t}| p^{x_0}_{T-t}(x) 
\le m_K,  \  \forall (t,x) \in K.$$
To show that left-hand side of \eqref{div_OU} is locally bounded on $[0,T[ \times \R^d$
it remains to show that $(t,x) \mapsto \int_{\R^d}  p^{x_0}_{T-t}(x) \nu(dx_0)$ is lower bounded on $K$. 
Indeed, let $I$ be a compact of $\R^d$. Since $(t,x,x_0) \mapsto p^{x_0}_{T-t}(x)$
is strictly positive and continuous is lower bounded by a constant $c(K,I)$.
The result follows choosing $I$ such that $\nu(I) > 0$.

 To conclude, it remains to show that the functions $(t,x) \mapsto \frac{{\mathop {div_x}}\left(\Sigma\left(T-t\right)_{i\cdot}p^{\nu}_{T-t}\left(x\right)\right)}{p^{\nu}_{T-t}\left(x\right)}, \ i \in [\![1,d]\!]$ defined on $[0,T[\times\R^d$ has locally bounded spatial derivatives, which implies that they are Lipschitz with linear growth on each compact of $[0,T[\times\R^d$.
By technical but easy computations, the result follows
using the fact the real functions $a \mapsto \vert a \vert^m \exp\left(-\frac{a^2}{2}\right)$, $m= 1,2,$
are bounded. 

\end{enumerate}

\end{proof}
\noindent We give now a global well-posedness result for the SDE
 \eqref{MKIntro}.
\begin{thm}\label{MKOU_WellP}
 ({\bf The McKean SDE: well-posedness in the Ornstein-Uhlenbeck case.})
  \begin{enumerate}
\item Suppose the initial condition $\mu$ equals ${\bf u}^\nu\left(T\right)$ for some $\nu \in \shp\left(\R^d\right)$. Then, equation \eqref{MKIntro} admits existence in law, strong existence, uniqueness in law and pathwise uniqueness. 
\item Otherwise \eqref{MKIntro} does not admit any solution.
\end{enumerate}
\end{thm}

\begin{proof}
Item 2. can be proved using similar arguments as for the proof of Proposition \ref{CP49}.
Let $(Y, {\bf p})$ be a solution of \eqref{MKIntro} and set $ \nu_0 = {\bf p}(T)$.
By Proposition \ref{PProbRep},  ${\bf p}\left(T-\cdot\right)$ is a
  solution of the PDE \eqref{EDPTerm0},
  so that ${\bf p}(T-\cdot)$  verifies also the PDE \eqref{Fokker} with initial value
  $ \nu_0$. Since, by Proposition \ref{FwdOU_Uniq}, uniqueness holds for
  \eqref{Fokker}, it follows that  ${\bf p}(T-\cdot) =  {\bf u}^{\nu_0}$
  which concludes the proof of item 2.

  We prove now item 1. For this, taking into account Proposition \ref{MKProp},
  Yamada-Watanabe theorem and related results for classical SDEs,
it suffices to show strong existence
and pathwise uniqueness. 
We set  $\shc := \shp\left(\R^d\right).$
\begin{description}
\item{a)} Concerning the strong existence statement, we want to apply Proposition \ref{MKEx_Prop} stated in the Appendix.
 Since $b,\sigma$ are affine,  Assumption \ref{Lip1d} trivially
holds; Property \ref{GH1}  with respect to $\shc$ thanks to Proposition \ref{FwdOU_Uniq}.
It remains to verify Property \ref{MKEx_1} with respect to $\shc$
and Property
\ref{MKEx_3} (in the Appendix).

By Lemma \ref{OU_lemma},
for all $t \in ]0,T]$,
${\bf u}^\nu(t)$
admits $p^\nu_t$ (see \eqref{Epnu}) for density. Then, relation \eqref{HP}
below holds since,
by  \eqref{Epnu} and  the considerations above, 
$(t,x) \mapsto p^\nu_t(x)$ is locally bounded with locally bounded spatial derivatives. Hence, Property \ref{MKEx_1} holds with respect to
$\shc$.
Finally, Lemma \ref{OU_lemma} implies that ${\bf u}^\nu\left(T-\cdot\right)$ belongs to $\sha_2$. Hence, Property \ref{MKEx_3} holds with respect to $\shc$
and so  Proposition \ref{MKEx_Prop} implies
existence in law.

\item{b)} Let $\left(Y,{\bf p}\right)$ be a solution of equation \eqref{MKIntro}. Proposition \ref{PProbRep} implies that ${\bf p}\left(T-\cdot\right)$ solves \eqref{EDPTerm0}. Then, Proposition \ref{FwdOU_Uniq} gives ${\bf p}\left(T-\cdot\right) = {\bf u}^{\nu_0}$ with $\nu_0 = {\bf p}\left(T\right)$. Lemma \ref{OU_lemma} implies ${\bf p}$ belongs to $\sha_2$.
\item{c)} It remains to show pathwise uniqueness in $\sha_2$
  for which we will make use of Corollary \ref{Coro}, lying on 
Property \ref{APDETerm}, both stated in the Appendix.
Indeed we check that Property \ref{APDETerm} holds with respect
to
$\shc$ thanks to Theorem \ref{BwdOU_Uniq}. Now, point 2 of Corollary \ref{Coro} implies pathwise uniqueness in $\sha_2$ since $b,\sigma$ are locally Lipschitz with linear growth in space. 
\end{description}

\end{proof}


\section{Appendix}

\setcounter{equation}{0}

For ease of reading the paper, we have postponed some technical results in this appendix. Sections
\ref{MKEX} and \ref{MKUNIQ} link the well-posedness of the PDE \eqref{EDPTerm0}
to the well-posedness of the McKean SDE \eqref{MKIntro}.
In particular Proposition \ref{MKEx_Prop} (resp. Corollary \ref{Coro})
links the existence (resp. uniqueness)
of the PDE \eqref{EDPTerm} with the SDE \eqref{MKIntro}.
Sections \ref{sec:ProofLemma1} and \ref{sec:ProofLemma2} give the proofs of two technical Lemma (Lemma~\ref{Lemma} and \ref{FriedAr}).

\subsection{PDE with terminal condition and existence 
for the McKean SDE} \label{MKEX}
\noindent 

We suppose that Property \ref{GH1} is in force for a 
fixed $\mathcal{C} \subseteq \mathcal{P}\left(\R^d\right)$ and consider the Property \ref{MKEx_1}
 with respect to $\shc$ and Properties
 \ref{MKEx_2} and  \ref{MKEx_3} related
to a given function $ {\bf u}: [0,T] \rightarrow \shm_+(\R^d)$.
\begin{property} \label{MKEx_1}
  \noindent
\begin{enumerate}
\item ${\bf u}\left(0\right)$ belongs to $\mathcal{C}$.
\item $\forall t \in ]0,T[$, ${\bf u}\left(t\right)$ admits a density
 with respect to the Lebesgue measure on $\R^d$ (denoted by $u\left(t,\cdot\right)$) and for all $t_0 > 0$ and all
 compact $K \subset \mathbb{R}^d$
\begin{equation} \label{HP}
\int^{T}_{t_0}\int_{K} \left|u\left(t,x\right)\right|^2 + \sum^{d}_{i=1}\sum^{d}_{j=1}\left|\sigma_{ij}\left(t,x\right)\partial_{i}u\left(t,x\right)\right|^2dxdt < \infty.
\end{equation}
\end{enumerate}
\end{property}
\begin{rem} \label{R45}
\noindent Suppose Assumption \ref{Lip1d} holds and let ${\bf u}$ be a measure-valued function verifying Property \ref{MKEx_1}. Then 
 \eqref{HP} implies that the family of densities $u\left(T-t,\cdot\right), t \in ]0,T[$
 verifies condition \eqref{IdInt} appearing in Definition \ref{MKSol}. 
  To show this, it suffices to check that for all $t_0 > 0$, all compact $K \subset \R^d$ and all $\left(i,j,k\right) \in [\![1,d]\!]^2\times[\![1,d]\!]$
\begin{equation}\label{Integr} 
\int^{T}_{t_0}\int_{K}\left|\partial_j\left(\sigma_{ik}\left(s,y\right)\sigma_{jk}\left(s,y\right)u\left(s,y\right)\right)\right|dyds < \infty.
\end{equation}
 The integrand
 appearing in \eqref{Integr} is well-defined. Indeed, in the sense of distributions we have
\begin{equation}\label{Deriv}
\partial_j\left(\sigma_{ik}\sigma_{jk}u\right) = \sigma_{ik}\sigma_{jk}\partial_ju + u\left(\sigma_{ik}\partial_j\sigma_{jk} + \sigma_{jk}\partial_j\sigma_{ik}\right);
\end{equation}
moreover the components of $\sigma$ are Lipschitz, so 
they are (together with their space derivatives) locally bounded. Also $u$
and $ \sigma_{jk}\partial_j u$ are square integrable by \eqref{HP}, which
implies \eqref{Integr}.

\end{rem}
We introduce two other properties possibly fulfilled by a function
${\bf u}: [0,T] \rightarrow \shm_+\left(\R^d\right)$.

\begin{property} \label{MKEx_2}
	${\bf u}(T)$ admits a density
  and $\restriction{{\bf u}\left(T-\cdot\right)}{[0,T[\times\R^d}$ belongs to $\sha_1$.  
\end{property}
\begin{property} \label{MKEx_3}
${\bf u}(T)$ admits a density and $\restriction{{\bf u}\left(T-\cdot\right)}{[0,T[\times\R^d}$ belongs to $\sha_2$.
\end{property}
 We remark that Property \ref{MKEx_3} implies  \ref{MKEx_2}.

\begin{prop} \label{MKEx_Prop}
 Suppose the validity of Assumptions \ref{Lip1d}. 
We also suppose that the backward PDE \eqref{EDPTerm0} with terminal condition $\mu$ admits at least
 an $\shm_+\left(\R^d\right)$-valued solution ${\bf u}$ in the sense of Definition \ref{Def}, fulfilling
Property \ref{GH1} 
 and Property \ref{MKEx_1} with respect to $\shc$.
Then \eqref{MKIntro} admits existence in law in $\sha_{\shc}$. 

Moreover, if ${\bf u}$ fulfills 
Property \ref{MKEx_2}  (resp. \ref{MKEx_3})
then \eqref{MKIntro} admits existence in law in $\sha_{\shc}\cap\sha_1$
 (resp. strong existence in $\sha_{\shc}\cap\sha_2$). 
\end{prop}
\begin{proof} 
  \noindent 
Let ${\bf u}$ the function of the statement such that 
fulfilling Property \ref{MKEx_1}, i.e.
${\bf u}\left(0\right)$ belongs to $\mathcal{C}$ 
We consider now
a filtered probability space $\left(\Omega, \shf, \left(\mathcal{F}_t\right)_{t\in[0,T]}, \P\right)$ equipped with an $\left(\mathcal{F}_t\right)_{t\in[0,T]}$-Brownian motion $W$. Let $X_0$ be a r.v. distributed according to ${\bf u}(0)$.
Under Assumption \ref{Lip1d}, it is well-known that there is a solution  $X$ to 
\begin{equation}\label{SDE}
X_t = X_0 + \int^{t}_{0}b\left(s,X_s\right)ds + \int^{t}_{0}\sigma\left(s,X_s\right)dW_s, \ t \in [0,T].
\end{equation}
\smallbreak
\noindent Now, by Proposition \ref{PFundam}, $t \mapsto \mathcal{L}\left(X_t\right)$ is a $\shp\left(\R^d\right)$-valued solution
of the PDE \eqref{Fokker} in the sense of \eqref{weakbis} with initial value ${\bf u}\left(0\right)\in\mathcal{C}$. 
Then Property \ref{GH1} for {\bf u} implies
\begin{equation} \label{MKIdLaw}
\mathcal{L}\left(X_t\right) = {\bf u}\left(t\right), \ t \in [0,T],
\end{equation}
 since ${\bf u}$ solves also the PDE \eqref{Fokker} with initial value ${\bf u}\left(0\right)\in\mathcal{C}$. This implies in particular that ${\bf u}$ is probability valued and that for all $t\in ]0,T[$, $X_t$ has $u\left(t,\cdot\right)$
as a density fulfilling
condition \eqref{HP}.

 Combining this observation with Assumption \ref{Lip1d}, Theorem 2.1 in \cite{haussmann_pardoux} 
states that there exists a filtered probability space
$\left(\Omega,\shg, (\shg_t)_{t\in[0,T]}, \Q\right)$ equipped with some
Brownian motion $\beta$ and a copy of  $\hat X$ (still denoted by the same letter) 
such that $\widehat{X}$ fulfills the first line of the SDE \eqref{MKIntro} with $\beta$ and 
\begin{equation} \label{Eup}
{\bf p}\left(t\right) := {\bf u}\left(T-t\right), \ t \in ]0,T[.
\end{equation}
Finally, existence in law for the SDE \eqref{MKIntro} in the sense of Definition \ref{MKSol} holds since  $(\widehat{X}, {\bf u}\left(T-\cdot\right))$ is a solution of 
\eqref{MKIntro} on the same filtered probability space and the same Brownian
motion above. 
This  occurs in $\sha_{\shc}$ since $\mathcal{L}\left(\widehat{X}_T\right) \in \shc$ thanks to 
equality \eqref{MKIdLaw} for $t = T$.

We discuss rapidly the {\it moreover} point.
\begin{itemize}
\item Suppose that {\bf u} fulfills
  Property \ref{MKEx_2}. Then ${\bf u}\left(T-\cdot\right)$ belongs to 
$\sha_\shc \cap \sha_1$ and we also have existence in law in $\sha_{\shc}\cap\sha_1$.
\item  Suppose that {\bf u} fulfills  
   Property \ref{MKEx_3}. Then,
  taking into account \eqref{Eup},
 strong existence and 
pathwise uniqueness for the
first line of \eqref{MKIntro}
  holds by classical arguments since the coefficients are locally
Lipschitz with linear growth, see \cite{RevuzYorBook}   Exercise (2.10),
and Chapter IX.2 and \cite{RevuzYorBook}, Th. 12.
section V.12. of \cite{rogers_v2}. 
By Yamada-Watanabe theorem this implies uniqueness in law, which shows 
that ${\bf u}\left(T-\cdot\right)$ constitutes the marginal laws
of the considered strong solutions.
This concludes the proof of strong existence
  in $\sha_{\shc}\cap\sha_2$ 
since ${\bf u}\left(T-\cdot\right)$ belongs to 
$\sha_\shc \cap \sha_2$, by Property \ref{MKEx_3}.
\end{itemize}
\end{proof}
\begin{rem} \label{RExistence2}
By \eqref{Eup},  the second component {\bf p}
of the solution of \eqref{MKIntro} is given by 
${\bf u}\left(T-\cdot\right).$

\end{rem}

\subsection{PDE with terminal condition and uniqueness for the McKean SDE} \label{MKUNIQ}

 In this subsection we discuss some questions related to uniqueness
for the PDE \eqref{MKIntro}. We consider the following Property 
related to
a given subset $\mathcal{C}$ of $\mathcal{P}\left(\R^d\right)$.

\begin{property} \label{APDETerm}
 The PDE \eqref{EDPTerm0} with terminal condition $\mu$ admits at most a
 $\shp\left(\R^d\right)$-valued solution ${\bf u}$ in the sense of Definition \ref{Def} such that ${\bf u}\left(0\right)$ belongs to $\mathcal{C}$.
\end{property}
 We recall that Section \ref{S32} provides various classes of examples 
where Property \ref{APDETerm} holds.
\begin{prop} \label{MKProp}
   Suppose the validity of Property  \ref{APDETerm} with respect to  $\shc$
   and suppose $b,\sigma$ to be locally bounded. 
 
 Let $\left(Y^i, {\bf p}^i\right), \ i \in \{1,2\}$ be two solutions of 
the SDE \eqref{MKIntro} in the sense of Definition \ref{MKSol} such that ${\bf p}^1\left(T\right), {\bf p}^2\left(T\right)$ belong to $\shc$.
 Then,
\be
{\bf  p}^1 = {\bf p}^2.
\ee
\end{prop}
\begin{proof} \ 
 Proposition \ref{PProbRep} shows that ${\bf p}^1\left(T-\cdot\right), {\bf p}^2\left(T-\cdot\right)$ are $\shp\left(\R^d\right)$-valued solutions of the PDE
 \eqref{EDPTerm0} in the sense of Definition \ref{Def} with terminal value $\mu $. Property \ref{APDETerm} gives the result since ${\bf p}^1\left(T\right), {\bf p}^2\left(T\right)$ belong to $\shc$.
\end{proof}
 As a corollary, we establish some consequences about uniqueness in law and
 pathwise uniqueness results for the SDE \eqref{MKIntro} in the classes $\sha_1$ and $\sha_2$. 

\begin{corro} \label{Coro}
   Suppose the validity of Property \ref{APDETerm} with respect to $\shc$.
   Then, the following results hold.
\begin{enumerate}
\item If $b$ is locally bounded,  $\sigma$ is continuous and if
  the non-degeneracy Assumption \ref{Zvon3} holds
    then the SDE 
\eqref{MKIntro} admits uniqueness in law in $\sha_{\shc}\cap\sha_1$.
  \item  If $b,\sigma$ are locally Lipschitz with linear growth in space, then \eqref{MKIntro} admits pathwise uniqueness in $\sha_{\shc}\cap\sha_2$.
\end{enumerate}
\end{corro}
\begin{proof}
  \noindent
  If $\left(Y,{\bf p}\right)$ is a solution of the SDE \eqref{MKIntro} and is such that
  ${\bf p}\left(T\right) $ belongs to $\shc$, then by  Proposition \ref{MKProp}
 ${\bf p}$ is determined by
  $\mu = \mathcal{L}\left(Y_0\right)$.

  To show that item 1. (resp. 2.) holds, it suffices to show that the classical SDE
\begin{equation} \label{FrozenSDE}
  dX_t =  \left(b(t,X_t; {\bf p}_t) -
    \widehat{b}(t,X_t)\right)
    dt + \widehat{\sigma}\left(t,X_t\right)dW_t, \ t \in [0,T[,
\end{equation}
where $b$ was defined in \eqref{EBP} and
 $W$ an $m$-dimensional Brownian motion, admits
uniqueness in law (resp. pathwise uniqueness).
The mentioned uniqueness in law 
is a consequence of Theorem 10.1.3 in \cite{stroock} and pathwise uniqueness 
 holds by \cite{RevuzYorBook} Exercise (2.10),
and Chapter IX.2 and \cite{rogers_v2} Th. 12.
Section V.12.
\end{proof}

\subsection{Proof of Lemma \ref{Lemma}}
\label{sec:ProofLemma1}

\begin{proof} 

\noindent For a given $\left(x,y\right) \in \R^d \times\R^d$ we set
\be
Z^{x,y}_t := X^{y}_t - X^{x} _t, t\in[0,T].
\ee
We have 
\begin{equation} \label{EZxy}
Z^{x,y}_t = y-x + \int^{t}_{0}B^{x,y}_r Z^{x,y}_rdr +\sum^{m}_{j=1} \int^{t}_{0}C^{x,y,j}_r Z^{x,y}_rdW^j_r, \ t\in[0,T],
\end{equation}
\noindent with, for all $r\in[0,T]$ 
\be
B^{x,y}_r := \int^{1}_{0}Jb\left(r,aX^{y}_r+(1-a)X^{x}_r\right)da, \quad
C^{x,y,j}_r  := \int^{1}_{0}J\sigma_{.j}\left(r,aX^{y}_r+(1-a)X^{x}_r\right)da, \forall \ j \in [\![1,m]\!].
\ee
 By the classical existence and uniqueness theorem for SDEs with Lipschitz coefficients we know that
\begin{equation} \label{SQI}
\mathbb{E}(\sup_{s \leq T} \left|X^{z}_s\right|^2) < \infty, 
\end{equation}
 for all $z \in \R^d$.
This implies
\begin{equation} \label{sup}
\mathbb{E}(\sup_{t\in[0,T]}\left| Z^{x,y}_t \right|^2) < \infty.
\end{equation}
 Now, It\^{o}'s formula gives, for all $t \in [0,T]$
\begin{equation} \label{ItoSquareNorm}
\left|Z^{x,y}_t\right|^2 = \left|y-x\right|^2 + 2\int^{t}_{0}\left<B^{x,y}_rZ^{x,y}_r,Z^{x,y}_r\right>dr + \sum^{d}_{j=1}\int^{t}_{0}\left|C^{x,y,j}_rZ^{x,y}_r\right|^2dr + 2\sum^{d}_{i=1}M^{x,y,i}_t,
\end{equation}
 where, for a given $i \in [\![1,d]\!]$, $M^{x,y,i}$ denotes the local martingale $\int^{\cdot}_{0}Z^{x,y,i}_s\sum^{d}_{j=1}\left(C^{x,y,j}_sZ^{x,y}_s\right)_idW^j_s$.
\smallbreak 
\noindent Consequently, for all $i \in [\![1,d]\!]$, we have
\begin{align} \label{EMForm}
\sqrt{\left[M^{x,y,i}\right]_T} &{}= \sqrt{\sum^{d}_{j=1} \int^{T}_{0}\left(Z^{x,y,i}_r\right)^2\left(C^{x,y,j}_r Z^{x,y}_r\right)^2_idr}, \nonumber \\
&{}\leq \sqrt{\sum^{d}_{j=1} \int^{T}_{0}\left|C^{x,y,j}_r Z^{x,y}_r\right|^2\left|Z^{x,y}_r\right|^2dr}, \\
&{}\leq \sqrt{T\sum^{d}_{j=1} \left(K^{\sigma,j}\right)^2}\sup_{r\in[0,T]}\left|Z^{x,y}_r\right|^2.
\nonumber
\end{align}
 By the latter inequality and \eqref{sup}, we know
that $\mathbb{E}\left([M^{x,y,i}]_T^{\frac{1}{2}}\right) < \infty$, so for all
$i \in [\![1,d]\!]$, $M^{x,y,i}$ is a true  martingale.
Taking expectation in identity \eqref{ItoSquareNorm}, we obtain
\smallbreak
\noindent 
\be
\mathbb{E}\left(\left|Z^{x,y}_t\right|^2\right) = \left|y-x\right|^2 + \int^{t}_{0}\mathbb{E}\left(2\left<B^{x,y}_rZ^{x,y}_r,Z^{x,y}_r\right> + \sum^{d}_{k=1}\left|C^{x,y,k}_rZ^{x,y}_r\right|^2\right)dr.
\ee
 Hence, thanks to Cauchy-Schwarz inequality and to the definition of $K^b$ and $K^{\sigma,j}$ for all $j \in [\![1,d]\!]$
\be
\mathbb{E}\left(\left|Z^{x,y}_t\right|^2\right) \leq \left|y-x\right|^2 + K \int^{t}_{0}\mathbb{E}\left(\left|Z^{x,y}_r\right|^2\right)dr
\ee
 and we conclude via  Gronwall's Lemma.
 \end{proof}

\subsection{Proof of Lemma \ref{FriedAr}}
\label{sec:ProofLemma2}

 Let $\nu \in \shp\left(\R^d\right)$.
For each given $t \in [0,T]$, we denote by $G_t$ the differential operator such that for all $f\in\mathcal{C}^2\left(\R^d\right)$
\be
G_tf = \frac{1}{2}\sum^{d}_{i,j=1}\partial_{ij}\left(\Sigma_{ij}\left(t,\cdot\right)f\right) - \sum^{d}_{i=1}\partial_i\left(b_i\left(t,\cdot\right)f\right).
\ee
 Assumption \ref{smoothness} implies that for a given $f \in \mathcal{C}^{2}\left(\R^d\right)$, $G_tf$ can be rewritten in the two following ways:
\begin{equation} \label{Friedman}
G_tf =  \frac{1}{2}\sum^{d}_{i,j=1}\Sigma_{ij}(t,\cdot)\partial_{ij}f +
 \sum^{d}_{i=1}(\sum^{d}_{j=1}\partial_i\Sigma_{ij}(t,\cdot)- b_i(t, \cdot))\partial_if + c^1(t,\cdot)f,
\end{equation}
 with
 $$ c^1 : (t,x) \mapsto \frac{1}{2}\sum^{d}_{i,j=1}\partial_{ij}\Sigma_{ij}(t,x) - \sum^{d}_{i=1}\partial_ib_i(t,x).$$
\begin{equation} \label{Aronson}
G_tf =  \frac{1}{2}\sum^{d}_{i,j=1}\partial_j (\partial_i\Sigma_{ij} (t,\cdot)f + \Sigma_{ij}(t,\cdot)\partial_if - \sum^{d}_{i=1}b_i(t,\cdot)\partial_if) - \sum^{d}_{i=1}\partial_ib_i(t,\cdot)f.
\end{equation}
\smallbreak
\noindent On the one hand, combining identity \eqref{Friedman} with
 Assumptions \ref{Zvon1a}, \ref{Zvon1b},
 \ref{Zvon3} and \ref{smoothness}, there exists a fundamental solution $\Gamma$ (in the sense of Definition stated in Section 1. p.3 of \cite{friedman_1964}) of $\partial_tu = G_tu$, thanks to Theorem 10. Section 6 Chap. 1. in the same reference. Furthermore, there exists $C_1,C_2 > 0$ such that for all $i \in [\![1,d]\!]$, $x,\xi \in \R^d$, $\tau \in [0,T]$, $t > \tau$,
\begin{equation} \label{PropFriedman_1}
\left|\Gamma\left(x,t,\xi,\tau\right)\right| \leq C_1\left(t-\tau\right)^{-\frac{d}{2}}\exp\left(-\frac{C_2\left|x-\xi\right|^2}{4\left(t-\tau\right)}\right),
\end{equation}
\begin{equation}\label{PropFriedman_2}
\left|\partial_{x_i}\Gamma\left(x,t,\xi,\tau\right)\right| \leq C_1\left(t-\tau\right)^{-\frac{d+1}{2}}\exp\left(-\frac{C_2\left|x-\xi\right|^2}{4\left(t-\tau\right)}\right),
\end{equation}
\noindent thanks to identities (6.12), (6.13) in Section 6 Chap. 1 in \cite{friedman_1964}.

On the other hand, combining Identity \eqref{Aronson} with Assumption \ref{smoothness}, there exists a so called weak fundamental solution $\Theta$ of $\partial_tu = G_tu$ thanks to Theorem 5 in \cite{AronsonGeneral}. In addition, there exists $K_1,K_2,K_3 > 0$ such that for almost every $x,\xi \in \R^d$ , $\tau \in [0,T]$, $t \geq \tau$ 
\begin{equation}\label{PropAronson}
\frac{1}{K_1}\left(t-\tau\right)^{-\frac{d}{2}}\exp\left(-\frac{K_2\left|x-\xi\right|^2}{4\left(t-\tau\right)}\right) \leq \Theta\left(x,t,\xi,\tau\right) \leq K_1\left(t-\tau\right)^{-\frac{d}{2}}\exp\left(-\frac{K_3\left|x-\xi\right|^2}{4\left(t-\tau\right)}\right), 
\end{equation}
 thanks to point (ii) of Theorem 10 in \cite{AronsonGeneral}. 

 Our goal is now to show that $\Gamma$ and $\Theta$ coincide. To this end, we adapt the argument developed at the beginning of Section 7 in \cite{AronsonGeneral}. Fix a function $H$ from $[0,T]\times\R^d$ belonging to $\mathcal{C}^\infty_c\left([0,T]\times\R^d\right)$. Identity (7.6) in Theorem 12 Chap 1. Section 1. of \cite{friedman_1964} implies in particular that the function 
$$ u: \left(t,x\right) \mapsto \int^{t}_{0}\int_{\R^d}\Gamma\left(x,t,\xi,\tau\right)
H\left(\tau,\xi\right)d\xi d\tau,$$
 is continuously differentiable in time, two times continuously differentiable in space and is a solution of the Cauchy problem
\begin{equation}\label{CauchyPb}
\begin{cases}
\partial_tu\left(t,x\right) = G_tu\left(t,x\right) + H\left(t,x\right),
 \ \left(t,x\right) \in ]0,T]\times\R^d, \\
u\left(0,\cdot\right) = 0. 
\end{cases}
\end{equation} 
It is consequently also a weak (i.e. distributional)
solution of \eqref{CauchyPb},
which belongs to  $\she^2(]0,T]\times\R^d)$ 
(see definition of that space in \cite{AronsonGeneral})
since $u$ is bounded thanks to inequality \eqref{PropFriedman_1}
and the fact that $H$ is bounded.
 Then, point (ii) of Theorem 5 in \cite{AronsonGeneral}
says that
$$ (t,x) \mapsto \int^{t}_{0}\int_{\R^d}\Theta\left(x,t,\xi,\tau\right)H\left(\tau,\xi\right)d\xi d\tau$$
is the unique weak solution in $\she^2(]0,T]\times\R^d)$ of \eqref{CauchyPb}. 
This implies that for every $(t,x) \in ]0,T] \times \R^d$ we have 
\be
\int^{t}_{0}\int_{\R^d}\left(\Gamma - \Theta\right)\left(x,t,\xi,\tau\right)H\left(\tau,\xi\right)
d\xi d\tau = 0.
\ee
 Point (i) of Theorem 5 in \cite{AronsonGeneral} 
(resp inequality \eqref{PropFriedman_1})  implies that $\Theta$ (resp. $\Gamma$) belongs to $L^{p}\left(]0,T]\times\R^d\right)$ as a function of $(\xi,\tau)$, 
for an arbitrary $p \geq d + 2 $. Then, we conclude that
for all $\left(t,x\right) \in ]0,T] \times \R^d$,  
 \begin{equation} \label{coincide}
\Theta\left(x,t,\xi,\tau\right) = \Gamma\left(x,t,\xi,\tau\right), \
d\xi d\tau a.e. 
\end{equation}
for all $(\tau,\xi) \in [0,t[ \times \R^d$.
This happens by density of $\mathcal{C}^\infty_c\left([0,T]\times\R^d\right)$
 in $L^{q}\left(]0,T]\times\R^d\right)$, $q$ being the conjugate of $p$.

This, together with \eqref{PropAronson} and the fact that
$\Gamma$ is continuous in $(\tau,\xi)$ implies that 
\eqref{PropAronson} holds for all $(\tau,\xi)  \in [0,t[ \times  \R^d$
and therefore
\begin{equation}\label{PropAronsonBis}
\frac{1}{K_1}\left(t-\tau\right)^{-\frac{d}{2}}\exp\left(-\frac{K_2\left|x-\xi\right|^2}{4\left(t-\tau\right)}\right) \leq \Gamma \left(x,t,\xi,\tau\right) \leq K_1\left(t-\tau\right)^{-\frac{d}{2}}\exp\left(-\frac{K_3\left|x-\xi\right|^2}{4\left(t-\tau\right)}\right).
\end{equation}
We introduce
$$
q_{t} := x \mapsto \int_{\R^d}
 \Gamma\left(x,t,\xi,0\right)\nu\left(d\xi\right).
$$
 By \eqref{PropAronsonBis}, with $\tau = 0$ we get
\begin{equation}
\label{PropFriedman_3}
q_{t}\left(x\right) \geq \frac{1}{K_1}t^{-\frac{d}{2}}\int_{\R^d}\exp\left(-\frac{K_2\left|x-\xi\right|^2}{4t}\right)\nu\left(d\xi\right).
\end{equation}
\smallbreak

\noindent We denote now by ${\bf v}^\nu$ the measure-valued mapping such that
 ${\bf v}^\nu\left(0,\cdot\right) = \nu$ and for all $t \in ]0,T]$, ${\bf v}^\nu\left(t\right)$ has density $q_{t}$ with respect to the Lebesgue measure on $\R^d$. We want to show that ${\bf v}^\nu$ is a solution of the PDE 
\eqref{Fokker} with initial value $\nu$ to conclude ${\bf u}^\nu = {\bf v}^\nu$ thanks to the 
 validity of Property \ref{GH1} because of Lemma \ref{LC313} and
 Assumptions \ref{Zvon1a}, \ref{Zvon1b} and \ref{Zvon3}.
 To this end, we remark that the definition 
of a fundamental solution for $\partial_tu = G_tu $
says that
$u$ is a $C^{1,2}$ solution and consequently also 
a solution in the sense of distributions.
In particular  for all $\phi \in \mathcal{C}^{\infty}_c\left(\R^d\right)$, for all $t \geq \epsilon > 0$
\begin{equation} \label{NearFP} 
\int_{\R^d}\phi\left(x\right){\bf v}^\nu\left(t\right)\left(dx\right) = \int_{\R^d}\phi\left(x\right){\bf v}^\nu\left(\epsilon\right)\left(dx\right) + \int^{t}_{\epsilon}\int_{\R^d}L_s\phi\left(x\right){\bf v}^\nu\left(s\right)\left(dx\right)ds.
\end{equation}
 To conclude, it remains to send $\epsilon$ to $0+$. 
Theorem 15 section 8. Chap 1. and point (ii) of the definition stated p. 27 in \cite{friedman_1964}
 imply in particular that
 for all $\phi \in \mathcal{C}^\infty_c\left(\R^d\right)$, $\xi \in \R^d$, 
\be
\int_{\R^d}\Gamma\left(x,\epsilon,\xi,0\right)\phi\left(x\right)dx \underset{\epsilon\to 0+}{\longrightarrow} \phi\left(\xi\right).
\ee
 Fix now $\phi \in \mathcal{C}^\infty_c\left(\R^d\right)$.
 In particular thanks to Fubini's theorem, \eqref{PropAronson} and
 Lebesgue's dominated convergence theorem we have
\begin{align*}
\int_{\R^d}\phi\left(x\right){\bf v}^\nu\left(\epsilon\right)\left(dx\right) &{}= \int_{\R^d}\phi\left(x\right)\int_{\R^d}\Gamma\left(x,\epsilon,\xi,0\right)\nu\left(d\xi\right)dx \\
&{} = \int_{\R^d}\int_{\R^d}\Gamma\left(x,\epsilon,\xi,0\right)\phi\left(x\right)dx\nu\left(d\xi\right) \\
&{} \underset{\epsilon \to 0+}{\longrightarrow}  \int_{\R^d}\phi\left(\xi\right)\nu\left(d\xi\right). 
\end{align*}
 By \eqref{NearFP} ${\bf v}^\nu$ is a solution
of the PDE \eqref{Fokker}
and consequently ${\bf u}^\nu = {\bf v}^\nu$, so that, for every $t \in ]0,T]$,
${\bf u}^\nu\left(t\right)$ admits $
u^\nu(t,\cdot) =  q_{t}$ for density with respect to the Lebesgue measure on $\R^d$. Now,
 integrating the inequalities \eqref{PropFriedman_1}, \eqref{PropFriedman_2} with respect to $\nu$ and combining this with inequality \eqref{PropFriedman_3}, we obtain
the existence of $K_1,K_2,C_1,C_2 > 0$ such that for all $t \in ]0,T]$, for all $x \in \R^d$, for all $i \in [\![1,d]\!]$
 \begin{equation*} 
 \frac{1}{K_1}t^{-\frac{d}{2}}\int_{\R^d} \exp\left(-\frac{K_2\left|x-\xi\right|^2}{4t}\right)\nu\left(d\xi\right)\leq u^\nu\left(t,x\right)\leq K_1t^{-\frac{d}{2}},
 \end{equation*}
 \begin{equation*}
 \left|\partial_iu^\nu\left(t,x\right)\right| \leq C_1t^{-\frac{d+1}{2}}.
 \end{equation*}
Consequently, the upper bounds in \eqref{dens} and \eqref{DerDens} hold.
Concerning the lower bound in \eqref{dens}, let $I$ be a compact subset of $\R^d$ such that
$\nu(I) > 0$, the result follows since $(t,x,\xi) \mapsto \exp\left(-\frac{K_2\left|x-\xi\right|^2}{4t}\right)$ is strictly positive, continuous and therefore
lower  bounded by a strictly positive constant on $K\times I$ for each compact $K$ of $]0,T]\times\R^d$.


\section*{Acknowledgments}

The authors are very grateful to the Referee for 
having read carefully the paper and  having contributed
to a significant improvement of the paper presentation.
The work was supported by a public grant as part of the
{\it Investissement d'avenir project, reference ANR-11-LABX-0056-LMH,
  LabEx LMH,}
in a joint call with Gaspard Monge Program for optimization, operations research and their interactions with data sciences.

\bibliographystyle{plain}
\bibliography{../../../../../BIBLIO_FILE/ThesisLucas}

\end{document}